\documentclass[smallcondensed]{svjour3}   
\usepackage{a4wide, amstext, amsmath, amssymb, graphicx, array, epsfig, psfrag}
\usepackage{amstext, amsmath, amssymb, graphicx}

\usepackage{mathscinet}
\usepackage{ifdraft}
\usepackage{epigraph}
\ifoptionfinal{
\usepackage[disable]{todonotes}
}{
\usepackage[bordercolor=white, color=white]{todonotes}
}

\makeatletter\providecommand\@dotsep{5}\def\listtodoname{List of Todos}\def\listoftodos{\hypersetup{linkcolor=black}\@starttoc{tdo}\listtodoname\hypersetup{linkcolor=blue}}\makeatother
\newcommand{\bfx}{\boldsymbol x}

\newcommand{\bfbeta}{\boldsymbol \beta}
\newcommand{\ele}{S}

\numberwithin{equation}{section}
\newcommand{\jump}[1]{[\![#1]\!]}

\begin{document}
%\numberwithin{equation}{section} 
\title{Weighted error estimates for transient transport problems
  discretized using continuous finite elements with interior penalty
  stabilization on the gradient jumps
}

\titlerunning{Weighted error estimates for transient transport problems
  discretized with CIP-FEM}        % if too long for running head

\author{Erik Burman
}


\institute{Erik Burman \at
Department of Mathematics, 
University College London, Gower Street, London, 
UK--WC1E  6BT, 
United Kingdom. \\
              \email{e.burman@ucl.ac.uk}           %  \\
}

\date{Received: date / Accepted: date}

\maketitle

\begin{abstract}
In this paper we consider the semi-discretization in space of a first
order scalar transport equation. For the space discretization we use
standard continuous finite elements with a stabilization consisting of
a 
penalty on the jump of the gradient over element faces.
We recall some global error estimates for smooth and rough solutions
and then prove a new local error estimate for the transient linear
transport equation. In particular we show that in the stabilized method the effect of
non-smooth features in the solution decay exponentially from the space
time zone where the solution is rough so that smooth features will be
transported unperturbed. Locally the $L^2$-norm error converges with the expected
order $O(h^{k+\frac12})$, if the exact solution is locally smooth. We then illustrate the results
numerically. In particular we show the good local accuracy in the
smooth zone of the stabilized method and that the standard Galerkin fails to
approximate a solution that is smooth at the final time if
underresolved features have been present in the solution at some time during
the evolution.
\end{abstract}

\section{Introduction}
\label{intro}
The discretization of transport problems has traditionally been dominated by
discontinuous Galerkin methods or finite volume methods, typically of
low order, since the continuous Galerkin method is known to have robustness
problems for first order partial differential equations (see
\cite[Chapter 5]{EG04}), or convection--diffusion
equations in the convection dominated regime. In certain situations
the use of high order continuous Galerkin methods is appealing, for
instance in the case of convection--diffusion equations, in particular
where the diffusion is nonlinear, or more complex situations such as large eddy simulation of turbulent flows, where
the pressure-velocity coupling can be decoupled using a pressure
projection method and the convective part handled explicitly. In
such situations, if continuous finite element spaces are used,
one must resort to a stabilized method to avoid a
  reduction of accuracy due to spurious oscillations. There is a very wide
literature on stabilized methods and for an overview of the topic
see for example \cite{EG21}. In the high order case, the
Spectral Vanishing Velocity method has been a popular choice
\cite{MT89,KK00,MAPS20}, but
other methods have also been designed to work in the high order case,
see the discussion in \cite{BQS08}.
In this work we will focus on the continuous interior penalty 
(CIP) stabilization, that was shown to allow for
close to $hp$-optimal error estimates in the high Peclet regime in
\cite{BE07}.
Recently \cite{MCBS21} this method was applied to under resolved simulations of
turbulent flows using high order polynomial approximation and shown to perform very well in this
context. Therein an eigenanalysis was performed which showed that the
CIP finite element method has similar advantageous dispersion
properties as the discontinuous Galerkin method (see also the report
\cite{CDBS20}) and in the
computations it was verified that its numerical dissipation was less
important than that of the spectral vanishing viscosity. 

Ideally stability of
the finite element method should match that of the continuous problem.
This is typically, by and large, true for elliptic pde, but much harder to achieve
in the hyperbolic case. Indeed, this would mean satisfaction
of a discrete maximum principle and stability and error estimates in
$L^1$. Both which typically remain open questions. Herein we will only
consider the stability in the $L^2$-norm for continuous finite element
approximations and linear symmetric stabilization of gradient penalty
type applied to the transient
scalar, linear first order equation. The analysis will mainly focus on semi discretization in space
on periodic domains, but the extension to the fully discrete case and
weakly imposed boundary conditions will be sketched.
The classical estimate for smooth solutions that is proven for
stabilized finite element methods is on the form
\begin{equation}\label{eq:error_stab}
\|(u- u_h)(\cdot,T)\|_\Omega \leq C(u) h^{k+{\frac12}}
\end{equation}
where $C(u)$ is a constant that depends on Sobolev norms of the exact
solution and on equation data, $h$ is the mesh-size and $k$ the
polynomial order. This estimate that is suboptimal by
  $h^{\frac12}$ is known to be sharp on general meshes \cite{PS95}
  (see also \cite{Bu05} for the sharpness of the estimate for the CIP method).
The continuous Galerkin method without stabilization, however,
only admits a bound of order $h^k$. The lost factor $h^{\frac12}$ is of
little consequence for smooth solutions, and high polynomial
order. However for low polynomial order or rough solutions it becomes
significant.
In section \ref{sec:error} below, we prove this type of error
estimate and some variations in weak norm for rough solutions. This
analysis uses ideas from \cite{BEF10,Bu14}.
Some remarks on the time discretization will
be added in subsection \ref{subsec:time}. In particular we will point out the situations
where the stabilization actually improves the stability of time stepping methods.

The estimate \eqref{eq:error_stab} is a weak result,
but it
 has become a proxy for stronger estimates
that give convergence also of the material derivative (see
\cite{Guer01,BG20} and Theorem \ref{thm:error_bounds} below) and importantly,
local estimates, using weighted norms, well known in the stationary
case \cite{JNP84,JSW87,Guz06,BGL09}. In the context of
time dependent problems such a weighted estimate takes the form
\begin{equation}\label{eq:weighted}
\|\varpi (u - u_h)(\cdot,T)\|_{\Omega} \leq C h^{k+\frac12} \left(\int_0^T
  \|\varpi D^{k+1} u\|_\Omega^2 ~\mbox{d} t \right)^{\frac12},
\end{equation}
where $D$ is a multi-index differential operator and the $\varpi$ is a weight function that is aligned with the characteristics
and decays exponentially away
from some zone of interest. This means that if $\varpi = 1$ in some zone
where the solution is smooth the influence of locally large
derivatives and underresolution at some
distance $d$ from this zone will be damped with a factor
$e^{-d/\sqrt{h}}$. 
We prove such an estimate in section \ref{sec:weight} for the space
semi-discretized stabilized formulation. To the best of my knowledge
there are no previous such estimates for continuous finite element
methods using symmetric stabilization. For earlier works on Streamline
Upwind Petrov-Galerkin methods (SUPG) in this
direction see \cite{Zhou95,FGN16}. The approach in \cite{Zhou95} relies strongly on
the space time finite element discretisation and an additional
artificial viscosity term and in \cite{FGN16} the
authors consider the SUPG method together with a first order backward
differentiation in time, on a form that can not easily be
extended to higher order time-discretizations. In neither case can the
arguments be applied independently of the time discretization. In this paper we apply the ideas
from \cite{BGL09} where weighted estimates were proved for the stationary
convection--diffusion equation with CIP-stabilization and \cite{BNO20},
where they were applied to an inverse boundary value problem subject
to a convection--diffusion equation. The result is presented for the
semi-discretized case only, but can be extended to standard stable time
discretizations. 
The results can also be extended to the case of
convection--diffusion equations with Neumann conditions on
the outflow boundary, by straightforward addition of the diffusive
terms and following the argument of \cite{BGL09}. 

In the numerical
section (section \ref{sec:numerics}) we will illustrate this localization
property of the error and show that
it is not shared by the standard (unstabilized) Galerkin finite element
method. Indeed, as we shall see, without stabilization Galerkin FEM fails to
approximate even smooth solutions satisfactory in case the solution
has had non-smooth features at any time during the computation. Indeed
it appears that the standard Galerkin method does not propagate
underresolved features of the solution
with the right speed, making it impossible for the
method to evacuate high frequency content from the computational
domain. For the stabilized
method on the other hand the weighted estimate \eqref{eq:weighted} guarantees that
smooth components of the solution are untainted by spurious high
frequency content at all times, since perturbations are damped
exponentially when crossing the characteristics.

\section{Model problem and finite element discretization}\label{sec:model}
We will discuss a first order hyperbolic problem in a periodic domain,
$\Omega = [-L,L]^n$, where $n \ge 1$ is the space dimension. Let
$\bfbeta \in C^0([0,T];[C^m(\bar \Omega)]^n)$,
$m\ge 1$, be a periodic vector field
satisfying $\nabla \cdot \bfbeta=0$ and consider
the first order hyperbolic problem
\begin{alignat}{1}\label{eq:1st_order}
\mathcal{L} u := \partial_t u + \bfbeta \cdot \nabla u & = f \quad\mbox{ in }
{(0,T) \times \Omega} \\
u(\cdot,0) & = u_0 \quad \mbox{ in } \Omega.
\end{alignat}
For smooth data $\bfbeta$, $u_0$ and $f$ there exists a unique solution by the
method of characteristics, but the problem admits a unique solution
also for more rough data \cite{GS10}. The solution satisfies the following regularity estimate
(a proof of this can be obtained after minor modifications of \cite[Lemma 2]{BOG20}),
\begin{equation}\label{eq:regularity}
\|u(t)\|_{H^j(\Omega)}
\leq C_\beta (\|f\|_{L^2((0,T);H^j(\Omega))} + \|u_0\|_{H^j(\Omega)}),
\quad t>0, \quad j \ge 0 \mbox{ when $m\ge j$}.
\end{equation}
Below we will always assume that $\bfbeta$ is smooth enough for
\eqref{eq:regularity} to hold.
The constant $C_{\beta}$ grows exponentially in time, with coefficient
dependent on the sup-norm of $\bfbeta$, and its derivatives of order up
to $j$. 
Below the notation $\beta_\infty = \sup_{x \in \bar \Omega}
|\bfbeta(x)|$ will be used. The $L^2$-norm over a domain $X \subset \Omega$ will be denoted by
$\|\cdot\|_X = (\cdot,\cdot)_X^{\frac12}$, where
$(\cdot,\cdot)_X^{\frac12}$ is the $L^2$-scalar product over $X$, also
$\|\cdot\|_{\infty}$ will denote the norm on $C^0(\bar \Omega)$.

Let $\{\mathcal{T}\}_h$ be a family of shape regular decomposition of
$\Omega$ in simplices $S$, 
$\mathcal{T} = \{ \ele \}$,
indexed by the (uniform) mesh size $h$. Let $\mathcal{F}$ denote
the set of faces of $\mathcal{T}$. $C$ will denote a generic constant that can have
different value at each appearance, but is always independent of the
mesh-parameter $h$.
Now define the finite element space
\[
V_h := \{v \in H_{per}^1(\Omega): v\vert_{\ele} \in \mathbb{P}_k(\ele), \mbox{
  for all } \ele \in \mathcal{T}\}
\]
where $\mathbb{P}_k(\ele)$ denotes the set of polynomials of degree less
than or equal to $k$ on $\ele$ and $H^1_{per}(\Omega)$ denotes the set of
periodic functions in $H^1$ on $\Omega$. We may then write a
semi-discretization in space, for $t>0$ find $u_h(t) \in V_h$, with
$u_h(0) = \pi_h u_0$, such
that
\begin{equation}\label{eq:standard_Galerkin}
(\mathcal{L} u_h(t),v_h)_\Omega =F(v_h), \quad \forall v_h \in V_h
\end{equation}
where $F(v_h) :=
(f,v_h)_\Omega$. Above $\pi_h$ denotes the $L^2$-projection onto
the finite element space $V_h$. For all $v \in L^2(\Omega)$, $\pi_h v
\in V_h$ satisfies
\[
(\pi_h v,w_h)_\Omega = (v,w_h)_\Omega, \forall w_h \in V_h.
\]
It is
well known that on locally quasi-uniform meshes the $L^2$-projection
satisfies the approximation bound,
\begin{equation}\label{eq:L2approx}
\|v - \pi_h v\|_\Omega + h \|\nabla (v - \pi_h v) \|_\Omega \leq C
h^{k+1} |v|_{H^{k+1}(\Omega)}, \quad \forall v \in H^{k+1}(\Omega).
\end{equation}

The formulation \eqref{eq:standard_Galerkin} defines a dynamical
system that admits a unique solution for $m \ge 0$ using standard techniques.
 Taking $v_h = u_h$ in \eqref{eq:standard_Galerkin}
and integrating in time
we see that \eqref{eq:standard_Galerkin} satisfies the bound
\eqref{eq:regularity} with $j=0$
\begin{equation}\label{eq:regularity_FEM}
\|u_h(t)\|_{\Omega}
\leq C_\beta \|f\|_{L^2((0,T);\Omega)} + \|u_0\|_{\Omega}, \quad t>0
\end{equation}
Since $\nabla \cdot \bfbeta = 0$ the bound holds with $C_\beta =
T^{\frac12}$.
Actually a stronger results holds for the $L^2$-norm when the norm on
$f$ is weakened. Indeed one may use that
\[
\int_0^T (f,u_h)_\Omega ~\mbox{d}t  + \|u_0\|_{\Omega}^2\leq \sup_{t \in (0,T)} \|u_h(t)\|_{\Omega} (\|f\|_{L^1((0,T);L^2(\Omega))}+\|u_0\|_{\Omega})
\]
to show that
\[
\sup_{t \in (0,T)} \|u_h(t)\|_{\Omega} \leq \|f\|_{L^1((0,T);L^2(\Omega))} + \|u_0\|_{\Omega}.
\]
However \eqref{eq:regularity} does not hold
for $u_h$ for $j=1$. A natural question to ask is then if the solution
to \eqref{eq:standard_Galerkin} gives any control of the
derivatives. In case $f \in L^2((0,T);\Omega)$ the immediate control
offered by \eqref{eq:1st_order} is $\mathcal{L} u\in
L^2((0,T);\Omega)$, that is the material derivative is bounded in
$L^2$. For \eqref{eq:standard_Galerkin} we get the corresponding 
bound $\pi_h \mathcal{L} u_h\in
L^2((0,T);\Omega)$. Since $\mathcal{L} u_h$ may be discontinuous over element
faces (due to the presence of derivatives in space) and $V_h \in C^0(\Omega)$, we see that $\pi_h \mathcal{L} u_h
\ne \mathcal{L} u_h $. It follows that not even this weakest
measure of derivatives of $u$ is controlled by
\eqref{eq:standard_Galerkin}. However since we are looking for control
in a discrete space we can use norm equivalence on discrete spaces in
the form of the inverse inequality \cite[Lemma 4.5.3]{BS08},
\begin{equation}\label{eq:inverse_ineq}
\|\nabla u_h\|_{\ele} \leq C h^{-1}
\|u_h\|_{\ele}
\end{equation}
 and observing that $\partial_t u_h \in V_h$, we see that
\begin{equation}\label{eq:Lstab}
\|\mathcal{L} u_h\|_\Omega \leq \|\pi_h \mathcal{L} u_h \|_\Omega +
C\beta_\infty h^{-1} \|u_h\|_\Omega.
\end{equation}
Combining \eqref{eq:Lstab} with the bound
\eqref{eq:regularity_FEM} 
\[
\|\mathcal{L} u_h\|_{L^2((0,T);\Omega)} \leq (1+C_\beta h^{-1}) (\|f\|_{L^2((0,T);\Omega)} + \|u_0\|_{\Omega}).
\]
So the constant in the control of the material derivative grows as
$O(h^{-1})$ under mesh refinement. Hence there is no improvement
compared to obtaining an $H^1$ estimate by combining the
$L^2$-stability of \eqref{eq:regularity_FEM} with \eqref{eq:inverse_ineq}.

The rationale for the addition of stabilized terms is to improve the
control of derivatives of $u_h$. As an example of stabilization terms
we here propose the gradient penalty term, introduced in \cite{DD76}
and shown to result in improved robustness and error estimates for convection dominated
flows in \cite{BH04},
\begin{equation}\label{eq:grad_jump}
s(w_h,v_h) = \sum_{F \in \mathcal{F}} \left<h_F^2 |\bfbeta|
\jump{\nabla u_h} ,\jump{\nabla v_h} \right>_F
\end{equation}
where $\left<u,v\right>_F = \int_F uv ~\mbox{d} s$, $\jump{\nabla v_h}\vert_F = \nabla v_h\vert_{F \cap \partial \ele_1}  \cdot
n_1 +\nabla v_h\vert_{F \cap \partial \ele_2}  \cdot n_2$ for $F = \bar \ele_1 \cap
\bar \ele_2$ and $n_1$ and $n_2$ denote the outward pointing normals of
the simplices $\ele_1$ and $\ele_2$ respectively. To reduce the amount of
crosswind diffusion the $|\bfbeta|$ factor may be replaced by  $|\bfbeta
\cdot n|$.  Define the stabilization semi norm by
\[
|w_h|_s := s(w_h,w_h)^{\frac12}.
\]
Also recall the following inverse inequality 
\begin{equation} \label{eq:stab_invest}
|w_h|_s \leq C h^{-\frac12} \beta_\infty^{\frac12} \|w_h\|_\Omega,
\quad \forall w_h \in V_h
\end{equation}
which is a consequence of the scaled trace inequality,
\cite[Theorem 1.6.6]{BS08}), 
\begin{equation}\label{eq:trace}
\|v\|_{\partial \ele} \leq C_{\ele} (h^{-\frac12} \|v\|_{\ele} + h^{\frac12}
\|\nabla v\|_{\ele}), \quad \forall v \in H^1(\ele)
\end{equation}
and \eqref{eq:inverse_ineq}.

The enhanced control of derivatives offered by this stabilization term
can be expressed as 
\begin{equation}\label{eq:stab_bound}
\inf_{v_h \in V_h} \|h^{\frac12}(\bfbeta \cdot \nabla u_h -
v_h)\|_\Omega^2 \leq C_s( \beta_\infty |u_h|^2_s + h \|\nabla \bfbeta\|^2_{\infty} \|u_h\|_\Omega^2).
\end{equation}
This is an immediate consequence of the local
estimate of \cite[Lemma 5.3]{BE07} and local approximation of $\bfbeta$
using lowest order Raviart-Thomas functions (for details see the discussion
\cite[Page 4]{BG20}). In particular this implies (since
$\partial_t u_h \in V_h$) that
\begin{equation}\label{eq:material_stab}
\|\mathcal{L} u_h\|_{\Omega} \leq C \|\pi_h \mathcal{L} u_h\|_\Omega +
 C^{\frac12}_s ( h^{-\frac12} \beta_\infty^{\frac12} |u_h|_s +  \|\nabla \bfbeta\|_{\infty} \|u_h\|_\Omega).
\end{equation}
It follows that when the finite element method has the additional
stability offered by the operator $s$, the constant in the bound for
$\mathcal{L} u_h$ will grow at the rate $O(h^{-\frac12})$ under mesh
refinement. Therefore we propose the stabilized method,
 find $u_h(t) \in V_h$, with
$u_h(0) = \pi_h u_0$, such
that
\begin{equation}\label{eq:stab_Galerkin}
(\mathcal{L} u_h(t),v_h)_\Omega +  \gamma s(u_h,v_h) = F(v_h), \quad \forall v_h \in V_h
\end{equation}
for $\gamma>0$. Clearly for $\gamma=0$ we recover the standard
Galerkin method.
\begin{remark}
Although we only consider continuous FEM below all the results
holds true for dG methods if the standard Galerkin
method (without stabilization) is replaced by the standard
dG method with central flux and the stabilized
finite element method is replaced by the standard dG method with upwind flux. There is indeed a common
misconception that the enhanced stability of the dG methods (space
discretization) is due to
the discontinuity of the element. The discontinuity only allows for
the improved control of the material derivative if there is
sufficent control on the solution jump. This can be introduced
through upwind fluxes, or otherwise. Indeed it is easy to see that the
upwind flux formulation is obtained from the central flux formulation
by adding the following stabilization term \cite{BMS04}
\[
s_{up}(v_h,w_h):= \frac12 \sum_{F \in \mathcal{F}}  \left<
|\bfbeta \cdot n_F| [v_h],[w_h] \right>_F
\]
where $[\cdot]$ simply denotes the jump of the function over the
element face $F$. In general the full jump needs to be penalized, but the minimal stabilization needed to make the dG method
satisfy the bound \eqref{eq:material_stab} depends on the mesh
geometry and the polynomial order \cite{BS07,WLZS19}. 
\end{remark}
\section{Stability estimate of the finite element method}\label{sec:stability}
Here we will formalize the discussion of the previous section to
obtain a stability estimate that will be useful for the subsequent
error analysis. 
%\[
%|||w_h|||^2 := s(w_h,w_h)\|h^{\frac12}
%\mathcal{L} u_h\|_\Omega^2 +  
%\]
First define the 
operator norms
\begin{equation}\label{eq:Fdef}
\|F\|_0 := \sup_{v_h \in V_h}
\frac{|F(v_h)|}{\|v_h\|_{\Omega}}  \mbox{ and } \|F\|_h := \sup_{v_h \in V_h}
\frac{|F(v_h)|}{\|v_h\|_{\Omega} + |v_h|_s}.
\end{equation}
With these definitions the arguments discussed in the previous section
may be written as follows.
\begin{theorem}\label{thm:stab_stab}
Let $u_h$ solve \eqref{eq:stab_Galerkin} with $\gamma>0$ then for all
$\tau \in [0,T]$
\[
\|u_h(\tau)\|_\Omega^2 + \gamma \int_0^\tau|u_h|_s^2 ~\mbox{d}t \leq
C_\beta \left(\int_0^\tau \|F\|^2_h~\mbox{d}t + \|u_h(0)\|_\Omega^2\right)
\]
where $C_\beta = O(\gamma^{-1}+T)$.
\end{theorem}
\begin{proof}
First take $v_h = u_h$ in \eqref{eq:stab_Galerkin} to obtain using the skew symmetry of the convective operator
\[
(\mathcal{L} u_h, u_h)_\Omega = \frac12 
\frac{d}{dt} \|u_h(t)\|_\Omega^2 
\]
and therefore after integration in time over $(0,\tau)$
\begin{multline*}
 \frac12 \|u_h(\tau)\|_\Omega^2 +\gamma \int_0^\tau |u_h(t)|_s^2 ~\mbox{d}t
 \leq \frac12 \|u_h(0)\|_\Omega^2+\int_0^\tau  F(u_h) ~\mbox{d}t\\
\leq \frac12 \|u_h(0)\|_\Omega^2+\int_0^\tau\|F\|_h (\|u_h(t)\|_{\Omega} + |u_h(t)|_s) ~\mbox{d}t .
\end{multline*}
Using the arithmetic-geometric inequality $ab \leq \tfrac12
a^2 + \tfrac12 b^2$ it follows that $\|F\|_h (\|u_h(t)\|_{\Omega} +
|u_h(t)|_s) \leq (\gamma^{-1} +T) \|F\|_h^2 + \tfrac12 T^{-1}
\|u_h(t)\|^2_{\Omega}+\gamma \tfrac12 |u_h(t)|_s^2$ leading to
\[
\|u_h(\tau)\|_\Omega^2 +\gamma \int_0^\tau |u_h(t)|_s^2 ~\mbox{d}t  \leq
\|u_h(0)\|_\Omega^2 +  (\gamma^{-1}+ T)
\int_0^\tau \|F\|_h^2 ~\mbox{d}t +
 \int_0^\tau T^{-1} \|u_h(t)\|^2_{\Omega}~\mbox{d}t.
\]
By Gronwall's inequality we have
\begin{multline*}
\|u_h(\tau)\|_\Omega^2 \leq \left(\exp{\int_0^\tau T^{-1} ~\mbox{d}t} \right)\left(\|u_h(0)\|_\Omega^2+  (\gamma^{-1}+T) \int_0^\tau\|F\|_h^2 ~\mbox{d}t \right)\\
\leq C \left(\|u_h(0)\|_\Omega^2 + (\gamma^{-1}+T)\int_0^\tau\|F\|_h^2 ~\mbox{d}t\right).
\end{multline*}
We may then bound
\begin{multline*}
\gamma \int_0^\tau |u_h(t)|_s^2 ~\mbox{d}t \leq \|u_h(0)\|_\Omega^2 +  (\gamma^{-1}+T)
\int_0^\tau \|F\|_h^2 ~\mbox{d}t +
 \int_0^\tau T^{-1} \|u_h(t)\|^2_{\Omega}~\mbox{d}t \\ \leq C\left(\|u_h(0)\|_\Omega^2 + (\gamma^{-1}+T)
\int_0^\tau\|F\|_h^2 ~\mbox{d}t \right)
\end{multline*}
which concludes the proof.
\end{proof}
For the material derivative we can prove the similar bound
\begin{corollary}\label{cor:material_stab}
Let $u_h$ solve \eqref{eq:stab_Galerkin} with $\gamma>0$ then there
holds
\[
\int_0^T \|h^{\frac12} \mathcal{L} u_h\|_\Omega^2 ~\mbox{d}t \leq
C_\beta \zeta(\gamma)^2 \left (\|u_h(0)\|_\Omega^2+
\int_0^T  
(h \|F\|_0^2 + (\beta_\infty + h \|\nabla \bfbeta\|^2_\infty T ) \|F\|_h^2 ) ~\mbox{d}t\right),
\]
where $\zeta(\gamma) = \gamma^{\frac12} + \gamma^{-\frac12}$.
\end{corollary}
\begin{proof}
\[
\int_0^T \|h^{\frac12} \mathcal{L} u_h\|_\Omega^2 ~\mbox{d}t =
\int_0^T  (\mathcal{L} u_h, h\pi_h \mathcal{L} u_h)_\Omega ~\mbox{d}t +
\int_0^T \|h^{\frac12}(I - \pi_h) \mathcal{L} u_h\|_\Omega^2
~\mbox{d}t = T_1 + T_2.
\]
To bound the term $T_1$ we use the formulation
\eqref{eq:stab_Galerkin} to obtain
\[
(\mathcal{L} u_h, h \pi_h \mathcal{L} u_h)_\Omega = F(h \pi_h \mathcal{L}
u_h) - \gamma s(u_h,h \pi_h \mathcal{L}
u_h). 
\]
For the first term on the right hand side we see that using the first
definition of \eqref{eq:Fdef} and the
stability of the $L^2$-projection there holds
\[
F(h \pi_h \mathcal{L}
u_h) \leq\|F\|_0 \|h \pi_h \mathcal{L}
u_h\|_\Omega\leq h^{\frac12} \|F\|_0 \|h^{\frac12} \mathcal{L} u_h\|_\Omega.
\]
For the second term we use \eqref{eq:stab_invest} and the $L^2$-stability
of the projection to get
\[
\gamma s(u_h,h \pi_h \mathcal{L} u_h) \leq \gamma
s(u_h,u_h)^{\frac12} s(h \pi_h \mathcal{L} u_h,h \pi_h \mathcal{L}
u_h)^{\frac12} \leq C \gamma \beta_\infty^{\frac12} |u_h|_s
\|h^{\frac12}  \mathcal{L} u_h\|_\Omega.
\]
Observe that in the last inequality a factor $h^{\frac12}$ is lost due to the
  application of \eqref{eq:stab_invest}.
Collecting these bounds we see that
\[
T_1 \leq \int_0^T (h \|F\|_0^2 + C^2 \gamma^2 \beta_\infty |u_h|^2_s + \frac12
\|h^{\frac12}  \mathcal{L} u_h\|_\Omega^2) ~\mbox{d} t.
\]
To bound $T_2$ we note that by the
definition of the $L^2$-projection
$\|h^{\frac12}(I - \pi_h) \mathcal{L} u_h\|_\Omega \leq
\|h^{\frac12}(\mathcal{L} u_h - v_h)\|_\Omega $ for all $v_h \in V_h$
and apply \eqref{eq:stab_bound} and the fact that $\partial_t u_h \in
V_h$, leading to
\[
T_2 = \int_0^T \inf_{v_h \in V_h} \|h^{\frac12}(\bfbeta \cdot \nabla u_h -
v_h)\|_\Omega^2 ~\mbox{d} t\leq C_s\int_0^T (\beta_\infty |u_h|^2_s + h \|\nabla \bfbeta\|^2_{\infty} \|u_h\|_\Omega^2) ~\mbox{d} t.
\]
The claim follows by the bounds on $T_1$ and $T_2$ and the result of Theorem \ref{thm:stab_stab}.
\end{proof}
\begin{remark}
Observe that the presence of both positive and negative powers of
$\gamma$ in $\zeta$, shows that the estimate degenerates both for
vanishing stabilization and for too strong stabilization. If $\gamma$
goes to inifinity the solution has to become $C^1$ and the solution
will in this case coincide with the standard Galerkin approximation in
the $C^1$-subspace, which is unstable, see
discussion in \cite{BQS10}
\end{remark}
\section{Error estimates for the stabilized formulation \eqref{eq:stab_Galerkin}}\label{sec:error}
Using the stability estimates of Theorem \ref{thm:stab_stab} it is
straightforward to derive the error estimate \eqref{eq:error_stab} for
smooth solutions. Below we will also use the Corollary \ref{cor:material_stab} to obtain
an optimal order $O(h^k)$ error estimate for the material derivative.

Then we will assume that $f \in L^2(0,T;\Omega)$ in \eqref{eq:regularity} so that
we only have $u \in L^2(0,T;\Omega)$. In this case we will show that
the stabilized finite element method still converges in a weaker norm.

\begin{theorem}\label{thm:error_bounds}
Let $u_0 \in H^{k+1}(\Omega)$, $f \in L^2(0,T;H^{k+1}(\Omega))$, let
$u$ be the solution of \eqref{eq:1st_order} and $u_h$ the solution of
\eqref{eq:stab_Galerkin}. Then there holds, for all $T>0$
\[
\|(u - u_h)(\cdot,T)\|_\Omega + \gamma \left(\int_0^T |u_h|_s^2 ~\mbox{d}t\right)^{\frac12} \leq C_\beta
\zeta(\gamma) h^{k+\frac12} (\|f\|_{ L^2(0,T;H^{k+1}(\Omega))} + \|u_0\|_{H^{k+1}(\Omega)})
\]
and
\[
\left(\int_0^T \|\mathcal{L}(u - u_h)\|_\Omega^2
  ~\mbox{d}t\right)^{\frac12} \leq C_\beta \zeta(\gamma)^2
h^{k} \|u\|_{ H^1(0,T;H^{k+1}(\Omega))},
\] 
where $\zeta(\gamma) := \gamma^{\frac12} + \gamma^{-\frac12}$ 
and $C_\beta$ depends on $\beta_\infty$ and $\|\nabla
\bfbeta\|_{\infty}$ and $T$.
\end{theorem}
\begin{proof}
This result is a consequence of the stability of Theorem
\ref{thm:stab_stab}, the consistency and \eqref{eq:stab_bound}. It is standard material (see \cite[Section
76.4]{EG21}) however for completeness we include the short proof.

Using standard approximation estimates there holds \cite[Lemma 5.6]{BE07} 
\begin{equation}\label{eq:L2_approx}
\|\beta_\infty^{\frac12} h^{-\frac12}(u - \pi_h u)\|_\Omega + |u- \pi_h u|_s \leq C \beta_\infty^{\frac12} h^{k+\frac12} |u|_{H^{k+1}(\Omega)}.
\end{equation}
Hence by applying a triangle inequality we only need to consider the discrete error $e_h = 
\pi_h u - u_h$. Injecting it in the equation \eqref{eq:1st_order} and using
\eqref{eq:stab_Galerkin} we see that
\[
(\mathcal{L} e_h, v_h) + \gamma s(e_h,v_h) = F_\pi(v_h)
\]
with $F_\pi(v_h) = (\partial_t (\pi_h u-u), v_h)_\Omega + (\bfbeta \cdot
\nabla  ( \pi_h u-u), v_h)_\Omega +  \gamma s(\pi_h u, v_h)$.
Applying Theorem \ref{thm:stab_stab} we see that
\[
\|e_h(T)\|_\Omega^2 +\gamma \int_0^T |e_h|_s^2
  ~\mbox{d}t\leq C_\beta \int_0^T\|F_\pi\|_h^2 ~\mbox{d}t + \|e_h(0)\|_\Omega^2.
\]
By the definition of  $u_h(0)$, $e_h(0)=0$. Since $ \partial_t \pi_h u
= \pi_h \partial_t u$ we have using $L^2$-orthogonality and
intergration by parts
\[
F_\pi(v_h) =(u - \pi_h u, \bfbeta \cdot
\nabla  v_h - w_h)_\Omega +  \gamma s(\pi_h u, v_h), \quad \forall w_h \in V_h.
\]
It now follows using the Cauchy-Schwarz inequality,
\eqref{eq:stab_bound} and \eqref{eq:L2_approx} and recalling that under the
regularity assumptions on data $u(t) \in H^2(\Omega)$, that
\begin{equation}\label{eq:Ferror}
\|F_\pi\|_h \leq C_\beta \zeta(\gamma) h^{k+\frac12} |u|_{H^{k+1}(\Omega)}.
\end{equation}
The first claim then follows after an application of
\eqref{eq:regularity}.

For the second inequality we apply Corollary \ref{cor:material_stab} to
see that, since $e_h(0) = 0$,
\begin{equation}\label{eq:Lpert}
\int_0^T \|h^{\frac12} \mathcal{L} e_h\|_\Omega^2 ~\mbox{d}t \leq C \zeta(\gamma)^2
\int_0^T    (h \|F_\pi\|_0^2 + (\beta_\infty + h \|\nabla
\bfbeta\|_\infty^2 T) \|F_\pi\|_h^2 ) ~\mbox{d}t.
\end{equation}
It follows that we only need to bound $F$ in the stronger topology
$\|\cdot\|_0$ to conclude. Using the
Cauchy-Schwarz inequality and the inverse inequalities
\eqref{eq:inverse_ineq} and \eqref{eq:stab_invest}
\begin{alignat*}{1}
F_\pi(v_h) & =(u - \pi_h u, \bfbeta \cdot
\nabla  v_h)_\Omega + \gamma s(\pi_h u, v_h) \\
&\leq C  \beta_\infty \|h^{-1} (u - \pi_h u)\|_\Omega \|v_h\|_\Omega
+ C \gamma h^{-\frac12} \beta_\infty^{\frac12} |\pi_h u|_s \|v_h\|_\Omega.
\end{alignat*} 
It follows from \eqref{eq:L2_approx} that
\begin{equation}\label{eq:Fzero}
\|F\|_0 \leq C_\beta (1 +\gamma) h^k |u|_{H^{k+1}(\Omega)}.
\end{equation}
Combining this bound for $\|F\|_0$ with the bound \eqref{eq:Ferror} in \eqref{eq:Lpert} we see that
\begin{equation}
\int_0^T \|h^{\frac12} \mathcal{L} e_h\|_\Omega^2 ~\mbox{d}t \leq
C_\beta \zeta(\gamma)^4 h^{2 k+1}
\int_0^T 
|u|^2_{H^{k+1}(\Omega)} ~\mbox{d}t
\end{equation}
and we conclude using the approximation bound
\[
\| \mathcal{L} (u - \pi_h u) \|_\Omega \leq C (h^{k+1} \|\partial_t
u\|_{H^{k+1}(\Omega)} + \beta_\infty h^k  \|
u\|_{H^{k+1}(\Omega)} )
\]
and the triangle inequality.

\end{proof}
\begin{remark}
Note that the error estimate on the material derivative is optimal
compared with the approximation properties of the finite element space.
In the corresponding analysis for \eqref{eq:standard_Galerkin} 
only $\|F\|_0$ may be used for the upper bound in Theorem
\ref{thm:stab_stab}, resulting in a bound that is suboptimal by
$O(h^{\frac12})$.
\end{remark}
\subsection{Rough solutions: convergence in weak norms}\label{sec:rough}
Assume now that we have $f \in L^2((0,T);\Omega)$ in
\eqref{eq:stab_Galerkin} and $u_0 \in L^2(\Omega)$. Then $u \in  L^2((0,T);\Omega)$ is the best
we can hope for, making the error estimates of Theorem
\ref{thm:error_bounds} invalid. However if we estimate the error in
a weaker norm, we can still obtain an error bound with convergence
order, provided a stabilized method is used. For $\psi \in
H^1_{per}(\Omega)$ consider the adjoint
problem
\begin{alignat}{1}\label{eq:adjoint}
-\mathcal{L} \varphi & = 0\\
\varphi(\cdot,T) & = \psi .
\end{alignat}
This problem admits a unique solution and by \eqref{eq:regularity}
\begin{equation}\label{eq:stab_adjoint}
\sup_{t \in (0,T)} \|\varphi(t)\|_{H^1(\Omega)} \leq C_\beta
\|\psi\|_{H^1(\Omega)}.
\end{equation}
Let $V:=  H^1_{per} (\Omega)$ and introduce the dual norm
\[
\|v\|_{V'} := \sup_{w \in V\setminus 0} \frac{\left<v,w\right>_{V',V}}{\|w\|_{V}}
\]
where $\left<v,w\right>_{V',V}$ is a space duality pairing that we can
identify with the $L^2$-scalar product for $v \in L^2(\Omega)$.
We now proceed using duality to prove an a posteriori bound
\begin{proposition}\label{prop:apost} (A posteriori error bound)
 Let
$u$ be the solution of \eqref{eq:1st_order} with $f \in L^2(0,T;\Omega)$
and $u_0 \in L^2(\Omega)$ and $u_h$ the solution of
\eqref{eq:stab_Galerkin}, with $\gamma \ge 0$. Then there holds, for all $T>0$
\begin{alignat*}{1}
\frac{ ((u - u_h)(\cdot,T),\psi)_\Omega }{\|\psi\|_{V}}&\leq C_\beta h \| u_0 - \pi_h
u_0\|_\Omega \\ 
& + C_\beta\int_0^T (\inf_{v_h \in V_h}  h  \|f - \bfbeta \cdot
  \nabla u_h - v_h\|_\Omega + \gamma h^{\frac12} |u_h|_s )~\mbox{d}t.
\end{alignat*}
\end{proposition}
\begin{proof}
Using the adjoint equation and integration by parts we see that for
any $\psi \in H^1_{per}(\Omega)$,
\begin{alignat*}{1}
((u - u_h)(\cdot,T),\psi)_\Omega & = ((u - u_h)(\cdot,T),\psi)_\Omega +\int_0^T (u -
u_h,-\mathcal{L} \varphi)_\Omega ~\mbox{d}t \\& =  (u_0 - \pi_h u_0,
\varphi(\cdot,0))_\Omega + \int_0^T (\mathcal{L} (u -
u_h), \varphi)_\Omega ~\mbox{d}t  \\
&= (u_0 - \pi_h u_0,
(I- \pi_h) \varphi(\cdot,0))_\Omega\\
& + \int_0^T ((\mathcal{L} (u -
u_h), \varphi - \pi_h \varphi)_\Omega +\gamma  s(u_h,\pi_h \varphi))~\mbox{d}t.
\end{alignat*}
Considering the terms of the right hand side we see that
\[
(u_0 - \pi_h u_0,
(I- \pi_h) \varphi(\cdot,0))_\Omega \leq C h \|u_0 - \pi_h
u_0\|_\Omega \|\nabla \varphi(\cdot,0)\|_\Omega,
\]
\[
 ((\mathcal{L} (u -
u_h), \varphi - \pi_h \varphi)_\Omega \leq C h \inf_{v_h \in V_h} \|f - \mathcal{L} 
u_h - v_h\|_\Omega  \|\nabla \varphi\|_\Omega = C h \inf_{v_h \in V_h} \|f - \bfbeta \cdot
  \nabla u_h  - v_h\|_\Omega  \|\nabla \varphi\|_\Omega
\]
and
\[
 s(u_h,\pi_h \varphi) \leq |u_h|_s h^{\frac12} \beta^{\frac12}_\infty \|\nabla \varphi\|_\Omega.
\]
It follows that
\begin{multline*}
(u_0 - \pi_h u_0,
(I- \pi_h) \varphi(\cdot,0))_\Omega + \int_0^T ((\mathcal{L} (u -
u_h), \varphi - \pi_h \varphi)_\Omega +\gamma  s(u_h,\pi_h
\varphi))~\mbox{d}t\\
 \leq C \left(h \|u_0 - \pi_h u_0\|+ \int_0^T (\inf_{v_h \in V_h}  h  \|f - \bfbeta \cdot
  \nabla u_h - v_h\|_\Omega + \gamma \beta_\infty^{\frac12} h^{\frac12} |u_h|_s )~\mbox{d}t\right) \sup_{t
    \in  (0,T)} \|\varphi(t)\|_{H^1(\Omega)}.
\end{multline*}
We end the proof by applying the stability \eqref{eq:stab_adjoint}.
\end{proof}
\begin{remark}
A posteriori error estimates in negative norms for stationary first order pde was
introduced in \cite{HMSW99} and the case of transient problems using
stabilized FEM in \cite{Bu14}.
Observe that this a posteriori error estimate can not in general be sharp, indeed
for a smooth solution, by Theorem \ref{thm:error_bounds} we get
$O(h^{k+1})$ convergence in the dual norm. This follows by observing
that since we may take $v_h=\partial_t u_h$ and $f = \mathcal{L} u$,
\[
\inf_{v_h \in V_h}  h  \|f - \bfbeta \cdot
  \nabla u_h - v_h\|_\Omega \leq h  \|\mathcal{L} (u - u_h)\|_\Omega 
\] 
and then applying the second bound of Theorem \ref{thm:error_bounds}. We see that compared to
the $L^2$-estimate we have lost another power $h^{\frac12}$. Sharp
residual type a posteriori error estimates in the $L^2$-norm for transport
equations in dimension $>1$, so far 
to the best of my knowledge, have only been obtained under a saturation
assumption and using a stabilized finite element method, or a dG
method with upwind flux \cite{Bu09}.
\end{remark}
\begin{theorem} (A priori error estimate for rough solutions)
Let $u$ be the solution of \eqref{eq:1st_order} with $f \in L^2(0,T;L^2(\Omega))$
and $u_0 \in L^2(\Omega)$ and $u_h$ that of
\eqref{eq:stab_Galerkin} with $\gamma>1$. Then there holds
\[
\sup_{t \in [0,T)} \|(u-u_h)(\cdot,t)\|_{V'} \leq C_\beta (\zeta(\gamma)+1) h^{\frac12}
(\|f\|_{L^2(0,T;L^2(\Omega))} + \|u_0\|_{\Omega}),
\]
with $\zeta(\gamma) = \gamma^{\frac12} + \gamma^{-\frac12}$.
\end{theorem}
\begin{proof}
By definition 
\[
\|u-u_h\|_{V'} =  \sup_{w \in V\setminus 0} \frac{(u - u_h,w)_\Omega  }{\|w\|_{V}}.
\]
Applying Proposition \ref{prop:apost} we see that, after a Cauchy-Schwarz
inequality in time, for any $T>0$,
%\begin{alignat*}{1}
\[
\|(u-u_h(\cdot,T)\|_{V'} \leq  C_\beta h \| u_0 - \pi_h
u_0\|_\Omega  + C_\beta h^{\frac12} T^{\frac12} \left(\int_0^T (\inf_{v_h \in V_h}  h  \|f - \bfbeta \cdot
  \nabla u_h - v_h\|^2_\Omega + \gamma^2 |u_h|^2_s
  )~\mbox{d}t\right)^{\frac12}.
\]
%\end{alignat*}
Then noting that by \eqref{eq:stab_bound} there holds
\[
\inf_{v_h \in V_h}  h  \|f - \bfbeta \cdot
  \nabla u_h - v_h\|^2_\Omega \leq  h \|f\|_\Omega^2 +C_s( |u_h|_s^2 +
  h \|\nabla \bfbeta\|^2_\infty \|u_h\|^2_\Omega)
\]
we see that all the a posteriori terms depending on $u_h$ are either on
the form $|u_h|_s$ or on the form $\|u_h\|^2_\Omega$ and we conclude
by applying Theorem \ref{thm:stab_stab}.
\end{proof}
\subsection{Time discretization and stabilized methods}\label{subsec:time}
As a rule of thumb any time integrator with non-trivial imaginary
stability boundary extending into the complex plane will be stable and accurate in the sense \eqref{eq:error_stab},
 under a CFL condition depending on $\bfbeta$ and $\gamma$.
In particular any time discretization method allowing for a time discrete version of
an energy estimate of the
type in Theorem \ref{thm:stab_stab} may be applied and will lead to
optimal error estimates similar to those above. This includes all
A-stable schemes, backward
differentiation methods of first and second order, the Crank-Nicolson
method. Explicit methods with good stability properties such as
explicit strongly stable Runge-Kutta (RK) methods of order higher
than, or equal to, 3 are stable \cite{ZS04,ZS10,BEF10,XSZW19,XSZ20}.
Similar stability results are expected to hold for Adams-Bashforth
(AB) methods of order 3, 4, 7, 8 under standard hyperbolic CFL, 
$\delta t \leq Co\, h$, where $\delta t$ denotes the timestep and $Co$
the Courant number. See for instance \cite{HV03} for a discussion of
time-discretization of advection--diffusion equation, 
\cite{GFR15} for a discussion of the stability boundaries of AB
methods and \cite{BG20} for numerical experiments using AB3. All these methods
are energy stable regardless of whether or not stabilization is added.  The second order
RK method is energy stable under hyperbolic CFL only for
piecewise affine approximation and with added stabilization of the
form \eqref{eq:grad_jump} \cite{BEF10} (for dG FEM and affine approximation upwind
stabilization must be added \cite{ZS04}). In the general case (no stabilization,
higher polynomial approximation) the RK2 method is
stable only under a slightly more strict $CFL$ condition, indeed one
needs to assume $dt \leq Co\, h^{\frac43}$, with $Co$ fixed, but small
enough. This condition is the same for both cG and dG methods (see
\cite{ZS04,BEF10}). Recently an
analysis of the second order backward differentiation formula and the
Crank-Nicolson method (AB2) with
convection extrapolated to second order from previous time steps was
proposed for the discretization of \eqref{eq:stab_Galerkin}
\cite{BG20}. It was shown that these schemes are stable under similar
conditions as the RK2 scheme.
Such multi step schemes are particularly appealing in the context of IMEX methods for
convection--diffusion and hence provide a one-stage alternative to the RK2 IMEX
method analysed in \cite{BE12}.

\section{Weighted error estimates}\label{sec:weight}
In this section we will consider the slightly more technically
advanced case of weighted estimates. The idea is to show that
stabilization makes information follow the characteristics similarly as in
the physics. This means that for solutions with a localized
sharp layer, the 
dependence of a local error in the smooth zone on the regularity of the exact solution decreases
exponentially with the distance to the singularity. Hence a locally
large gradients in the solution can not
destroy the solution globally. This is not the case for approximations
produced using cG FEM without stabilization. These results touch at
the very essence of stabilized FEM, unfortunately their proofs are
quite technical and therefore these results in my opinion have
received less attention than they deserve. Here we try to give the 
simplest possible exposition of these ideas, without striving for optimality of
exponential decay or generality of meshes. We let the domain be infinite ($L =
\infty$) and let $u_0$ have compact support. To simplify the
discussion assume that $\bfbeta\equiv e_x$, where $e_x$ is the Cartesian unit vector in the
$x$-direction, so that $\bfbeta \cdot
\nabla u = \partial_x u$. Since here $\beta_\infty = 1$, below the 
dependence on the speed will not be tracked. First the
  case of a globally smooth solution will be considered (Theorem
  4). The objective is to obtain an estimate for the error in some
  subdomain $\Omega_0(t) \subset \Omega$ defined as
\[
\Omega_0(t) := \{\bfx \in \Omega: |\bfx_0 + \bfbeta t - \bfx  | < r_0 \}
\]
for some $\bfx_0 \in \Omega$ and some $r_0>0$.
The derivatives of $u$ are assumed to be moderate in a neighbourhood
of $\Omega_0$ and we will prove that the accuracy in this subdomain 
is independent of large derivatives in other parts of the domain, provided
they are sufficiently far away, relative to the mesh size. This
is achieved using weights so that the effect of portions of the domain
where locally the Sobolev norm is large decays exponentially with the
distance to $\Omega_0$. Then we will show how the arguments of the
smooth case can be used to prove accuracy in $\Omega_0$ in the case where the solution is locally
only $L^2$ in the far field (Corollary 2). The key message is that the
local accuracy of the approximation depends only on the local
smoothness of the exact solution and that perturbations due to
roughness in the solution is exponentially damped, except along characteristics. Finally we will discuss how
the arguments can be extended to bounded domains with weakly imposed
boundary condition and time discretization.

% and only consider the case $k=1$,
%i.e. piecewise affine approximation. 

Let $\varphi\in C^{k+1}(\Omega)$ be a smooth positive function defined using polar/spherical coordinates,
depending only on $r = |\bfx_0 - \bfx|$, with $\varphi'(r) \leq 0$, $\varphi(r) = 1$,
$r\leq r_0$, $\varphi(r) \sim \exp(-(r-r_0)/\sigma)$, $r>r_0$, with $\sigma = K\sqrt{h}$, $K>1$, and for some $C>0$,
\[
|\partial_r^l \varphi(r)| \leq C \sigma^{-l} \varphi(r), \quad l \ge 1.
\]
\begin{remark}
For the case $k=1$ we only require $\varphi \in C^1(\Omega)$. An
example of such a function with $r_0=1$ and $\sigma=5$ is given in
Fig. \ref{fig:bump} for illustration.
\begin{figure}
\begin{center}
\includegraphics[width=0.5\textwidth]{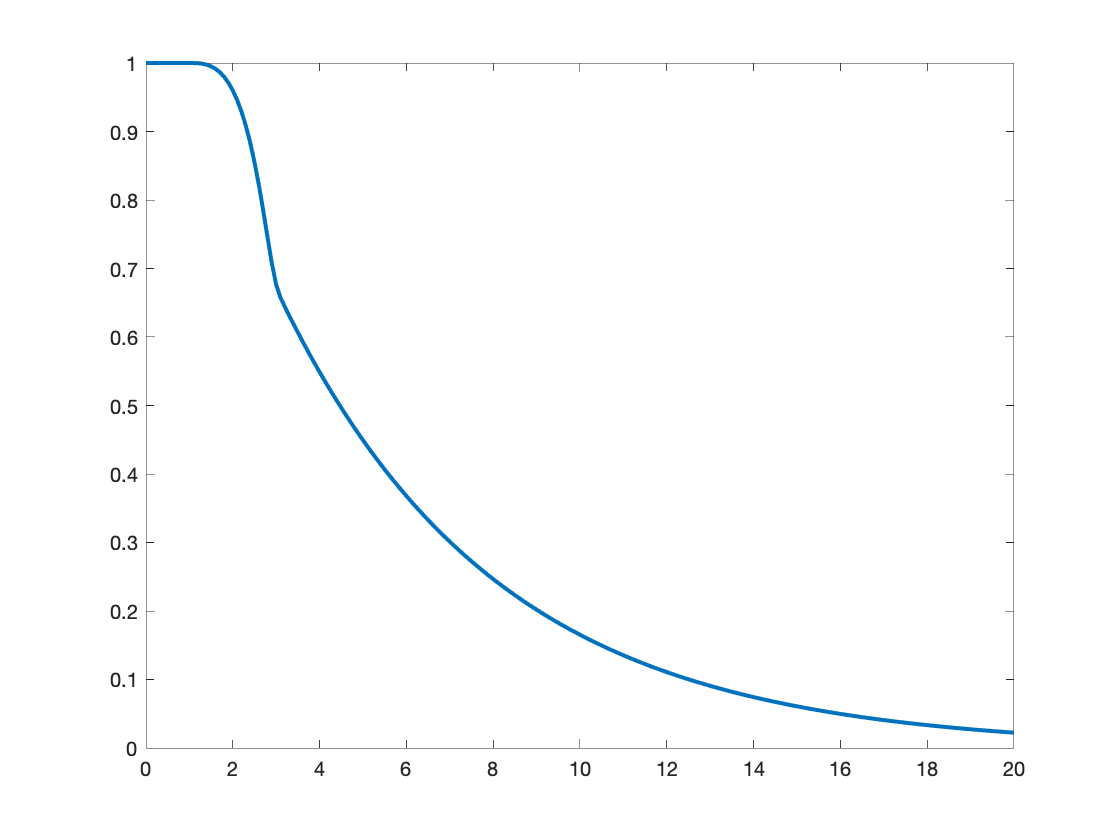} 
\end{center}
\caption{Example of the radial cross section of $\varphi \in
  C^1(\Omega)$ with $r_0=1$ and $\sigma=5$.}\label{fig:bump}
\end{figure}
\end{remark}

%\[
%\varphi(r) = \left\{ \begin{array}{ll} 1 & r\leq r_0\\
%g(r_0 - r) + (1 - g(r_0 - r)) \exp((r_0 - r)/\sigma) & r_0
%                                                                 < r <
%                                                                 r_0
%                                                                 + \sqrt{\sigma}
%\exp((r_0 - r)/\sigma) &  r \leq r_0
%\end{array} \right. 
%\]
%where $\sigma = K\sqrt{h})$ and the Gaussian function $g$ is given by
%\[
%g(r) = \exp(s(r)/(s(r)-1))
%\]
%where $s(r) = r^2/\sigma$.
%  r=0:0.1:20
% e=5 (=(2 K \sqrt{h}))
%  s = r.^2./e
%  g0 = exp(s./(s.*(s<1)-1)).*(s<1)
%  g1 = exp(-r/e)
% 
Define $\varpi(\bfx) = \varphi(\bfx - \bfbeta t)$ then, since $\varpi$
follows the characteristics $\mathcal{L}
\varpi = 0$,
and
\begin{equation}\label{eq:varpi_bound}
|D^l\varpi| \leq C \sigma^{-l} \varpi, \quad l \ge 1
\end{equation}
where the derivatives are taken with respect to space or time.
The objective is to prove stability and error estimates in the weighted norm
\[
\|v\|_\varpi:= \|\varpi v\|_\Omega.
\]
The same notation will be used occasionally below with different weight functions.
The rationale for the design of the weight function is that for all $v
\in L^{\infty}(0,T;L^2(\Omega))$ with $\mathcal{L} v \in
L^2(0,T;\Omega)$, by partial integration in space and
time,
\[
\int_0^T (\partial_t v  , \varpi^2 v)_\Omega = \|v(\cdot,
T)\|_\varpi^2 -  \|v(\cdot,
0)\|_\varpi^2 - \int_0^T (v  , \partial_t\varpi^2 v+\varpi^2 \partial_t v)_\Omega ~\mbox{d}t
\]
and
\[
(\bfbeta \cdot \nabla v, \varpi^2 v)_\Omega = - (v,(\bfbeta \cdot \nabla \varpi^2)
v +\varpi^2 \bfbeta \cdot \nabla v)_\Omega,
\]
there holds
\begin{alignat*}{1}
\int_0^T (\mathcal{L} v, \varpi^2 v)_\Omega ~\mbox{d}t & =
\|v(\cdot,T)\|_\varpi^2 -  \|v(\cdot,0)\|_\varpi^2\\
& - \int_0^T (v, \underbrace{(\mathcal{L}\varpi^2)}_{=0} v)_\Omega
+  (v, \varpi^2 \mathcal{L} v)_\Omega ~\mbox{d}t.
\end{alignat*}
Hence
\begin{equation}\label{eq:cont_weight_stab}
\int_0^T (\mathcal{L} v, \varpi^2 v)_\Omega ~\mbox{d}t =
\frac12\|v(\cdot,T)\|_\varpi^2 -  \frac12\|v(\cdot,0)\|_\varpi^2
\end{equation}
and therefore the following stability is satisfied by the continuous
equation, \eqref{eq:1st_order}, $\forall \sigma>0$,
\begin{equation}
\frac12\|u(\cdot,T)\|_\varpi^2 \leq  \frac12 \|u(\cdot,0)\|_\varpi^2 +
\int_0^T \|f\|_{\varpi}  \|u\|_\varpi ~\mbox{d}t
\end{equation}
from which we conclude
\[
\sup_{t \in (0,T)} \|u(\cdot,t)\|_\varpi \leq \|u(\cdot,0)\|_\varpi +
2 \int_0^T \|f\|_{\varpi}  ~\mbox{d}t.
\]
This relation expresses that the solution is transported along the
characteristics. The influence across
characteristics will be damped exponentially as
$\exp(-d/\sigma)$. However in the continuous case, since the bound
holds for all $\sigma>0$ the cut-off is sharp.

The aim is to make the error analysis for the solution of
\eqref{eq:stab_Galerkin} reproduce this type of
localization.
For the purposes of analysis we introduce the weighted stabilization operator 
\[
s_\varpi(v_h,w_h) = \sum_{F \in \mathcal{F}} \int_F h_F^2 \varpi^2 
\jump{\nabla v_h} \jump{\nabla w_h} ~\mbox{d} s,\mbox{ with semi-norm
} |w|_{s,\varpi} := s_\varpi(w,w)^{\frac12}
\]
and note that $s(v_h,\varpi^2 w_h) = s_\varpi(v_h,w_h)$.
Also recall the following weighted versions of \eqref{eq:stab_bound}
from  \cite[Lemma
  3.1, equation (3.1) and (3.2)]{BGL09}, here $\bfbeta_0\vert_{\ele} \in \mathbb{R}^n$ is some piecewise
  constant per element,
\begin{equation}\label{eq:stab_weight}
\|h^{\frac12} (\bfbeta_0 \cdot \nabla v_h
- \pi_h \bfbeta_0 \cdot \nabla v_h)\|^2_{\varpi} \leq C_{ws} ||\bfbeta_0| v_h|^2_{s,\varpi}
\end{equation}
and
\begin{equation}\label{eq:stab_weight2}
\|h^{\frac12} (\bfbeta \cdot \nabla  (\varpi^2 v_h)
- \pi_h(\bfbeta \cdot \nabla  (\varpi^2 v_h)))\|^2_{\varpi^{-1}} \leq C_{ws} |v_h|^2_{s,\varpi} + C_\beta
K^{-2} \|v_h\|^2_\varpi.
\end{equation}
The second bound differs from the bound in \cite {BGL09}, since there
the derivative of $v_h$ appears in the second term of the right hand side. The
proof however is similar. For completeness we detail it in Appendix.
We will need to use approximation in the weighted norm and therefore
collect some results on the $L^2$-projection in the following
Lemmas. The first one is taken from \cite{Bo06} and we refer to this
reference for the proof. The following two are
variations on results from \cite{BGL09} and for completeness we give
the proofs in Appendix. We note that all the above inequalities hold
both for the weight $\varpi$ and $\varpi^{-1}$, since by the construction of
the weight,
\[
|\nabla \varpi^{-1}| = |\varpi^{-2} \nabla \varpi| \leq C \varpi^{-2}
\sigma^{-1} \varpi = C 
\sigma^{-1} \varpi^{-1}.
\]
It follows that \eqref{eq:varpi_bound} is satisfied also for $\varpi^{-1}$.
\begin{lemma}(Stability $L^2$-projection)
Let $\pi_h$ denote the $L^2$-projection onto $V_h$. Then, if $\phi$ is a
function satisfying
\[
|\nabla \phi(x)| \leq \nu h^{-1} |\phi(x)|,
\]
for some $\nu>0$, sufficiently small then
there holds
\begin{equation}\label{L2_approx1}
\| \pi_h v\|_\phi \leq C\|v \|_\phi,
\end{equation}
\begin{equation}\label{L2_approx2}
\|\nabla\pi_h v\|_\phi \leq C \|\nabla v \|_\phi
\end{equation}
and
\begin{equation}\label{L2_approx3}
\|\nabla \pi_h v\|_\phi \leq C h^{-1} \|v\|_\phi , \quad \forall
v \in H^1(\Omega).
\end{equation}
\end{lemma}
\begin{proof}
The estimates \eqref{L2_approx1}-\eqref{L2_approx3} taken verbatim
from \cite[bounds (1.7) - (1.9)]{Bo06} (see also
\cite[Appendix]{EJ95}).
\end{proof}
The above stability estimates allows us to prove bounds on the
$L^2$-error in the weighted norm.
\begin{lemma}(Weighted approximation)\label{lem:weight_approx}
Let $\pi_h$ denote the $L^2$-projection onto $V_h$. Then for
$h^{\frac12}/K$ sufficiently small and $I_\delta = [t-\delta t, t+
\delta t]\cap [0,T]$ with $\delta t \in \mathbb{R}^+$,
$\delta t \sim h$, there holds
\begin{equation}\label{eq:move_varpi}
\max_{(x,t) \in \ele \times I_\delta} \varpi(x,t) \|v\|_{\ele} \leq 2  \min_{t
  \in I_\delta }\| v \varpi(\cdot,t)\|_{\ele}, \quad \forall v \in L^2(\ele),
\end{equation}
%\begin{equation}\label{eq:inverse_varpi}
%\|\nabla v_h\|_\varpi \leq C h^{-1}  \| v_h\|_\varpi, \forall v_h \in V_h,
%\end{equation}
\begin{equation}\label{eq:error_est}
\|(v - \pi_h v)\|_\varpi + h \|\nabla(v - \pi_h v)\|_\varpi \leq C h^{k+1}
\|D^{k+1} v\|_\varpi, \quad \forall v \in H^{k+1}(\Omega)
\end{equation}
and
\begin{equation}\label{eq:stab_est}
|v - \pi_h v|_{s,\varpi} \leq C h^{k+\frac12} \|D^{k+1} v\|_\varpi \quad \forall v \in H^{k+1}(\Omega).
\end{equation}
\end{lemma}
For the analysis we also need the following interpolation estimates on
weighted discrete functions.
\begin{lemma}(Super approximation)\label{lem:super}
Let $v_h \in V_h$. Assume that $h^{\frac12}/K$ is sufficiently small. Then there holds
\begin{equation}\label{super1}
\|\varpi^2 v_h - \pi_h (\varpi^2 v_h)\|_{\varpi^{-1}} + h \|\nabla (\varpi^2
v_h - \pi_h (\varpi^2
v_h))\|_{\varpi^{-1}}\leq C h^{\frac12}
K^{-1} \|v_h \|_{\varpi}
\end{equation}
and 
\begin{equation}\label{super2}
\left(\sum_{\ele \in \mathcal{T}} \|\varpi^{-1}\nabla(\varpi^2 v_h - \pi_h (\varpi^2 v_h))\|_{\partial
  \ele}^2 \right)^{\frac12} \leq C h^{-1} K^{-1} \|v_h\|_{\varpi}.
\end{equation}
\end{lemma}

We will now derive a weighted stability estimate for the finite
element formulation \eqref{eq:stab_Galerkin}. First use similar arguments
as for \eqref{eq:cont_weight_stab} to obtain for any $v_h \in C^1(0,T;V_h)$,
\[
\int_0^T (\mathcal{L} v_h, \varpi^2 v_h)_\Omega ~\mbox{d}t =
\frac12 \|v_h(\cdot,T)\|_\varpi^2 -  \frac12 \|v_h(\cdot,0)\|_\varpi^2
\]
and, since $\varpi \in C^1(\Omega)$ we see that
\[
s(v_h,\varpi^2 v_h) = |v_h|_{s,\varpi}^2
\]
Therefore, 
\begin{alignat}{1}\label{eq:first_weight}
\|v_h(\cdot,T)\|_\varpi^2  + 2 \gamma\int_0^T |v_h|_{s,\varpi}^2~\mbox{d}t
& = 2 \int_0^T( (\mathcal{L} v_h,\varpi^2  v_h
)_\Omega + \gamma s(v_h,\varpi^2  v_h) )~\mbox{d}t
%\\&
 +  \|v_h(\cdot,0)\|_\varpi^2.%\nonumber
\end{alignat}

However, since $\varpi^2 v_h\not \in V_h$ the equality
can not be used directly for the finite element formulation. We need to show
that stability similar to \eqref{eq:first_weight} can be obtained by testing by some interpolant of
$\varpi^2 v_h$. 
\begin{proposition}\label{prop:weight_stab}(Weighted stability)
Let $\gamma > 0$, $K>1$. Assume that $h^{\frac12}/K$ is sufficiently small. For all $v_h \in C^1(0,T;V_h)$ there holds 
\begin{alignat*}{1}
\|v_h(\cdot,T)\|_\varpi^2  + \gamma \int_0^T  |v_h|_{s,\varpi}^2
~\mbox{d}t   & \leq C/K^2  \int_0^T \|v_h\|_\varpi^2 ~\mbox{d}t \\
&+ 2 \int_0^T( (\mathcal{L} v_h, w_h)_\Omega + \gamma s(v_h,w_h) )~\mbox{d}t +  \|v_h(\cdot,0)\|_\varpi^2\nonumber
\end{alignat*}
where $w_h = \pi_h \varpi^2 v_h$ and the constant $C \sim \gamma + \gamma^{-1}$.
\end{proposition}
\begin{proof}
Starting from the equality \eqref{eq:first_weight} we add and subtract the finite element
formulation tested with some function $w_h$,
\begin{alignat}{1}\label{eq:second_weight}
\|v_h(\cdot,T)\|_\varpi^2  + 2 \gamma \int_0^T  |v_h|_{s,\varpi}^2 ~\mbox{d}t   & = 2\int_0^T( (\mathcal{L} v_h,
\varpi^2 v_h- w_h)_\Omega + \gamma s(v_h,\varpi^2 v_h - w_h) )~\mbox{d}t  \\
&+ 2 \int_0^T( (\mathcal{L} v_h, w_h)_\Omega + \gamma s(v_h,w_h) )~\mbox{d}t +  \|v_h(\cdot,0)\|_\varpi^2.\nonumber
\end{alignat}
We choose $w_h = \pi_h  (\varpi^2 v_h)$ to obtain, for
  an arbitrary $y_h \in V_h$
\begin{alignat*}{1}
(\mathcal{L} v_h,
\varpi^2 v_h - \pi_h (\varpi^2 v_h))_\Omega &= (\bfbeta \cdot \nabla v_h
- y_h, \varpi^2 v_h - \pi_h (\varpi^2 v_h))_\Omega \\[3mm]
&\leq \inf_{y_h \in V_h} \|h^{\frac12} (\bfbeta \cdot \nabla v_h
- y_h)\|_{\varpi} h^{-\frac12}\| (\varpi^2 v_h - \pi_h (\varpi^2 v_h))\|_{\varpi^{-1}}.
\end{alignat*}
Considering the stabilization term we see that
\[
s(v_h,\varpi^2 v_h -\pi_h (\varpi^2 v_h)) \leq |v_h|_{s,\varpi}
h \beta_\infty^{\frac12} \left(\sum_{F
  \in \mathcal{F}} \|\varpi^{-1} \jump{\nabla (\varpi^2 v_h -\pi_h (\varpi^2 v_h))}\|_F^2\right)^{\frac12}.
\]
Using the arithmetic-geometric inequality $ab \leq ((2 \epsilon)^{-1}
a^2 + (\epsilon 2^{-1}) b^2$, with $\epsilon =1$ or $2$, to split the terms in the
right hand side we obtain
\begin{alignat}{1}\label{eq:third_weight}
\|v_h(\cdot,T)\|_\varpi^2  + \frac74\gamma \int_0^T  s_\varpi(v_h,v_h)
~\mbox{d}t   & \leq \epsilon^{-1} \gamma^{-1} h^{-1}\int_0^T \underbrace{\|  (\varpi^2 v_h
- \pi_h (\varpi^2 v_h))\|_{\varpi^{-1}}^2}_{T_1} ~\mbox{d}t  \\
&+ \gamma h^2 \beta_\infty\int_0^T \underbrace{\sum_{F
  \in \mathcal{F}} \|\varpi^{-1} \jump{\nabla
  (\varpi^2 v_h -\pi_h (\varpi^2 v_h))}\|_F^2}_{T_2}~\mbox{d}t  \\
&+ \epsilon \gamma \int_0^T \underbrace{\inf_{y_h \in V_h} \|h^{\frac12} (\bfbeta \cdot \nabla v_h
- y_h)\|_{\varpi}^2}_{T_3}~\mbox{d}t  \\
&+ 2 \int_0^T( (\mathcal{L} v_h, w_h)_\Omega + \gamma s(v_h,w_h) )~\mbox{d}t +  \|v_h(\cdot,0)\|_\varpi^2.\nonumber
\end{alignat}

We need to bound the contributions $T_1$, $T_2$ and $T_3$ in terms of the quantities of the
left hand side and $\|v_h\|_\varpi$. Using \eqref{super1} immediately yields
\[
T_1 = \|  (\varpi^2 v_h
- \pi_h (\varpi^2 v_h))\|_{\varpi^{-1}}^2 \leq C K^{-2} h \|v_h\|_\varpi^2.
\]
By distribution of the integrals over the faces on simplices, splitting the jumps on the
contributions from the two sides and applying \eqref{super2} there holds
\begin{alignat*}{1}
T_2 \leq C\sum_{\ele \in \mathcal{T}}  \|\varpi^{-1} \nabla
  (\varpi^2 v_h -\pi_h (\varpi^2 v_h))\|_{\partial \ele}^2  \leq C/K^2
  h^{-2} \|v_h\|_\varpi^2.
\end{alignat*}
Finally for the term $T_3$ apply the weighted stabilization bound
\eqref{eq:stab_weight}, with $\bfbeta_0 \equiv e_x$, where $e_x$ is the Cartesian unit vector in the
$x$-direction
\[
T_3 = \inf_{y_h \in V_h} \|h^{\frac12} (\bfbeta \cdot \nabla v_h
- y_h)\|_{\varpi}^2 \leq C_{ws} |v_h|_{s,\varpi}^2.
\]
Collecting the bounds for $T_1$-$T_3$ and choosing $\epsilon = (2
C_{ws})^{-1}$ we see that
\begin{alignat}{1}\label{eq:final_stab_bound_proof}
\|v_h(\cdot,T)\|_\varpi^2  + \gamma \int_0^T   |v_h|_{s,\varpi}^2
~\mbox{d}t   & \leq (\gamma^{-1} + \gamma)  C/K^2  \int_0^T \|v_h\|_\varpi^2 ~\mbox{d}t \\
&+ 2 \int_0^T( (\mathcal{L} v_h, w_h)_\Omega + \gamma s(v_h,w_h) )~\mbox{d}t +  \|v_h(\cdot,0)\|_\varpi^2.\nonumber
\end{alignat}
\end{proof}
\begin{theorem}\label{weigh_error}
Assume that the hypothesis of Proposition \ref{prop:weight_stab} are
satisfied. Let $u \in L^\infty(0,T;H^{k+1}(\Omega))$ be the solution of \eqref{eq:1st_order} and $u_h$ the
solution of \eqref{eq:stab_Galerkin}. Then for all $T>0$ there holds
\[
\|(u - u_h)(\cdot,T)\|_{\varpi} \leq C_K h^{k+\frac12} \left(h \|D^{k+1} u
  (\cdot,T)\|_\varpi^2 + (\gamma + \gamma^{-1}) \int_0^T
  \|D^{k+1} u\|_\varpi^2 ~\mbox{d} t \right)^{\frac12}.
\]
The constant $C_K$ grows exponentially in time with coefficient
proportional to $(\gamma + \gamma^{-1}) K^{-2}$.
\end{theorem}
First note that we may split the error as $u - u_h = \underbrace{u -
\pi_h u}_{=-\eta} + \underbrace{\pi_h u - u_h}_{=e_h}$ and by \eqref{eq:error_est},
\[
\|(u - \pi_h u)(\cdot,T)\|_{\varpi}  \leq C h^{k+1} \|D^{k+1} u (\cdot,T)\|_\varpi.
\]
By the triangle inequality we only need to prove the bound on
$\|e_h(\cdot,T)\|_{\varpi}$.

Using the stability of Proposition \ref{prop:weight_stab} we see that, since
$e_h(\cdot,0) = 0$,
\begin{alignat}{1} \nonumber
\|e_h(\cdot,T)\|_\varpi^2  + \gamma \int_0^T  |e_h|_{s,\varpi}^2
~\mbox{d}t   & \leq C/K^2  \int_0^T \|e_h\|_\varpi^2 ~\mbox{d}t %\\
%&
+ 2 \int_0^T( (\mathcal{L} e_h, w_h)_\Omega + \gamma s(e_h,w_h) )~\mbox{d}t.
\end{alignat}
with $w_h = \pi_h (\varpi^2 e_h)$.
Now observe that the following consistency property holds
\[
\int_0^T (\mathcal{L}  (e_h - \eta), v_h)_\Omega -
\gamma s(u_h,
v_h) ~\mbox{d}t= 0,\quad \forall v_h \in V_h
\]
and hence
\[
\int_0^T( (\mathcal{L} e_h, w_h)_\Omega + \gamma s(e_h,w_h)
)~\mbox{d}t = \int_0^T( (\mathcal{L} \eta, w_h)_\Omega + \gamma s(\pi_h u_h,w_h)
)~\mbox{d}t.
\]
This leads to a perturbation equation on the form
\begin{alignat}{1}
\|e_h(\cdot,T)\|_\varpi^2  + \gamma \int_0^T  |e_h|_{s,\varpi}^2
~\mbox{d}t   & \leq C (\gamma + \gamma^{-1}) K^{-2}  \int_0^T \|e_h\|_\varpi^2 ~\mbox{d}t %\\
%&
+ 2 \int_0^T( (\mathcal{L} \eta, w_h)_\Omega + \gamma s(\pi_h u_h,w_h)
)~\mbox{d}t.\label{eq:pert_bound}
\end{alignat}

Considering the first term of the second integral in the right hand side we have using that
time derivation and the $L^2$-projection commute and the
$L^2$-orthogonality of $\eta$
\begin{alignat*}{1}
(\mathcal{L} \eta, w_h)_\Omega & = -(\eta, \bfbeta \cdot \nabla w_h -
y_h)_\Omega \leq h^{-\frac12} \|\eta\|_\varpi h^{\frac12} \inf_{y_h \in V_h} \|\bfbeta \cdot \nabla w_h -
y_h\|_{\varpi^{-1}} \\
&\leq h^{-1} \gamma^{-1}  C \|\eta\|_\varpi^2 + \frac14 \gamma
|e_h|^2_{s,\varpi} + C\gamma /K^2 \|e_h\|^2_\varpi.
\end{alignat*}
Here we used the inequality $ab \leq 4^{-1} a^2 + b^2$ and that by the triangle inequality followed by the bounds \eqref{super1},
and \eqref{eq:stab_weight2} there holds
\begin{alignat*}{1}
h^{\frac12} \inf_{y_h \in V_h}\|\bfbeta \cdot \nabla w_h -
y_h\|_{\varpi^{-1}} &\leq h^{\frac12} \|\bfbeta \cdot \nabla \pi_h (\varpi^2
e_h) - \bfbeta \cdot \nabla (\varpi^2
e_h)\|_{\varpi^{-1}} +
h^{\frac12} \inf_{y_h \in V_h}\|\bfbeta \cdot \nabla (\varpi^2
e_h) - y_h\|_{\varpi^{-1}} \\
&\leq h^{\frac12} \beta_\infty \|\nabla (\pi_h (\varpi^2
e_h) - \varpi^2
e_h)\|_{\varpi^{-1}} +  (C_{ws} |e_h|^2_{s,\varpi} + C_\beta
K^{-2} \|e_h\|^2_\varpi)^{\frac12} \\
&\leq C 
K^{-1} \|e_h\|_{\varpi}+  C_{ws} |e_h|_{s,\varpi} .
\end{alignat*}
For the last term in the right hand side of \eqref{eq:pert_bound} we have % observe that since $u(\cdot,t) \in H^{\frac32+\epsilon}(\Omega)$, $\epsilon>0$ we
\begin{alignat*}{1}
s(\pi_h u_h ,w_h)& =  s(\pi_h u_h,\pi_h (\varpi^2 e_h) - \varpi^2 e_h)
 + s(\pi_h u_h, \varpi^2 e_h) \\
& \leq C
|\pi_h u_h|^2_{s,\varpi}  + \frac14 
|e_h|^2_{s,\varpi}  \\
&+  h^2 \beta_\infty^2 \sum_{F \in \mathcal{F}} \|\varpi^{-1} \jump{\nabla
  (\varpi^2 e_h -\pi_h (\varpi^2 e_h))}\|_F^2 .
\end{alignat*}
Applying the bound \eqref{super2} to the last term in the right
hand side and collecting the estimates it follows that
\[
(\mathcal{L} \eta, w_h)_\Omega + \gamma s(\pi_h u_h,w_h)
\leq C (\gamma
|\pi_h u_h|^2_{s,\varpi} + h^{-1} \gamma^{-1}  \|\eta\|_\varpi^2)
+ \frac12\gamma
|e_h|^2_{s,\varpi}  + \gamma C/K^2 \|e_h\|_\varpi^2.
\]
Applying this bound in \eqref{eq:pert_bound} we have
\begin{alignat}{1} \nonumber
\|e_h(\cdot,T)\|_\varpi^2  + \frac12 \gamma\int_0^T |e_h|^2_{s,\varpi}
~\mbox{d}t   & \leq C(\gamma + \gamma^{-1})/K^2  \int_0^T \|e_h\|_\varpi^2 ~\mbox{d}t \\
&+ C \int_0^T(\gamma
|\pi_h u_h|^2_{s,\varpi} + h^{-1} \gamma^{-1}  \|\eta\|_\varpi^2)
~\mbox{d}t. \label{eq:pert2}
\end{alignat}
Since the solution is assumed regular, $u(\cdot,t) \in H^{\frac32+\epsilon}(\Omega)$,
$\epsilon>0$ we have $|\pi_h u_h|^2_{s,\varpi} = |\eta|^2_{s,\varpi}$.
Applying Lemma \ref{lem:weight_approx} yields
\[
\int_0^T(\gamma
|\eta|^2_{s,\varpi} + h^{-1} \gamma^{-1}   \|\eta\|_\varpi^2
)~\mbox{d}t\leq  C h^{2 k+1} (\gamma + \gamma^{-1}) \int_0^T\|D^{k+1} u\|^2_\varpi ~\mbox{d}t.
\]
The claim now follows by an application of Gronwall's inequality.

Consider the following subsets of $\Omega$, $\Omega_0(t) := \{x \in \Omega: \varpi(x,t) = 1\}$
and $\Omega_p(t) :=  \{x \in \Omega: \varpi(x,t) \leq h^p, p>0\}$. Then
denoting  $d
= \mbox{dist}(\Omega_0,\Omega_p)$ it follows by the construction of
$\varpi$ that 
\[
d \sim K p \sqrt{h} |\log (h)|,
\]
and the following bound holds
\begin{alignat*}{1}
\|(u-u_h)(\cdot,T)\|_{\Omega_0} \leq C  h^{k+\frac12}  (\|D^{k+1} u\|_{L^
\infty(0,T;L^2(\Omega  \setminus \Omega_p))} + h^{p}  \|D^{k+1} u\|_{L^\infty(0,T;L^2(\Omega_p))}).
\end{alignat*}
It follows that $D^{k+1}
u$ can be large in $\Omega_p$ without destroying the solution in
$\Omega_0$.
To apply the argument to $u_0$ that is not (globally) in $H^2(\Omega)$ one can
use the weighted $L^2$-stability in the error analysis above and still
obtain estimates. We present a sketch of this result in a Corollary
\begin{corollary}
Assume that the hypothesis of Proposition \ref{prop:weight_stab} are
satisfied.  Assume that $u \in L^\infty(0,T;L^2(\Omega)) \cap L^
\infty(0,T;H^{k+1}(\Omega  \setminus \Omega_p))$, with $p = k+1$ is the solution of
\eqref{eq:1st_order} and $u_h$ the solution of
\eqref{eq:stab_Galerkin}. Then there holds (omitting for
simplicity the dependence on $\gamma$).
\begin{alignat*}{1}
\|(u-u_h)(\cdot,T)\|_{\Omega_0} \leq C_K h^{k+\frac12} ( \| u\|_{L^
\infty(0,T;H^{k+1}(\Omega  \setminus \Omega_{p}))} + \|u\|_{L^\infty(0,T;L^2(\Omega_{p}))}).
\end{alignat*}
\end{corollary}
\begin{proof}
The proof follows that of Theorem \ref{weigh_error} closely. We only
need to substitute the $L^2$-projection for an interpolant with more
local properties before applying approximation. 
Let the domain $\Omega_{p,ih}(t)$ be defined by the union of all the
elements that intersect $\Omega_p(T)$ and an integer $i$ layers of nearest
neighbours.
The norm over $\Omega_{p,ih}(t)$ will be denoted $\|\cdot\|_{\Omega_{p,ih}}$.
Let $C_h$ denote the Cl\'ement interpolant defined using local
projections. It is well known
\cite[Lemma 1.127]{EG04} that if
for a given $\ele \in \mathcal{T}$, $\Delta_{\ele}$ denotes the set of
simplices sharing at least one vertex with $S$ and for a face $F$,
$\Delta_F$ denotes the set of simplices sharing at least one vertex
with $F$, then
\begin{equation}\label{eq:clement_approx}
\|v - C_h v\|_{H^m(\ele)} \leq C h^{l-m} \|v\|_{H^l(\Delta_{\ele})},
\quad\|v - C_h v\|_{H^m(F)} \leq C h^{l-m-\frac12}
\|v\|_{H^2(\Delta_F)}, \, 0 \leq m \leq l \leq k+1.
\end{equation} 
It is then straightforward
to use the approximation properties of $C_h$ in $\Omega  \setminus
\Omega_{p,1h}$ and the local stability of $C_h$ in $\Omega_{p,1h}$ to show the estimates
\begin{multline}\label{eq:clem_approx_L2}
\|(u - C_h u)(\cdot,t)\|_{\varpi} \leq C(  h^{k+1}  \|D^{k+1}
u(\cdot,t)\|_{\Omega  \setminus \Omega_{p}} + h^p
\|u(\cdot,t)\|_{\Omega_{p,2h}}) \\
\leq C h^{k+1} (\|u(\cdot,t)\|_{H^2(\Omega  \setminus \Omega_{p})} + \|u(\cdot,t)\|_{\Omega_{p}})
\end{multline}
and 
\begin{multline}\label{eq:clem_approx_s}
|C_h u(\cdot,t)|_{s,\varpi} \leq C( h^{k+\frac12} \|D^{k+1}
u(\cdot,t)\|_{\Omega  \setminus \Omega_{p}} + h^{-\frac12+p}
\|u(\cdot,t)\|_{\Omega_{p,2h}}) \\
\leq C h^{k+\frac12} (\|u(\cdot,t)\|_{H^{k+1}(\Omega  \setminus \Omega_{p})} + \|u(\cdot,t)\|_{\Omega_{p}}).
\end{multline}
For the second inequality we divide $|C_h u(\cdot,t)|_{s,\varpi}$ into
the sum over faces in $\Omega\setminus \Omega_{p,1h}$  and
$\Omega_{p,1h}$. The two different sets are treated differently. For faces in $\Omega\setminus \Omega_{p,1h}$ we proceeded as usual using that
$u(\cdot,t)\vert_{\Omega\setminus \Omega_p} \in
H^{\frac32+\epsilon}(\Omega\setminus \Omega_p)$ and apply
the local approximation properties on faces of $C_h$ (right inequality
of \eqref{eq:clement_approx}. For
faces in $\Omega_{p,1h}$ we can not use approximation and instead 
apply \eqref{eq:trace} and
\eqref{eq:inverse_ineq}. We also used that $\varpi\vert_{\Omega_{p,1h}} \leq C h^p$ by construction.  Observe that
by the weighted $L^2$-stability \eqref{L2_approx1} we have
\begin{equation}\label{eq:clem_stab}
\|(u - \pi_h u)(\cdot,T)\|_\varpi \leq C \|(u - C_h u)(\cdot,T)\|_\varpi 
\end{equation}
and hence as before we only need to prove the bound for
$\|e_h(\cdot,T)\|_\varpi $.
The inequality \eqref{eq:pert2} still holds. To conclude we observe
that using \eqref{eq:clem_stab}
\begin{equation}\label{eq:L2_clem_part}
\int_0^Th^{-1}  \|\eta\|^2_\varpi ~\mbox{d}t \leq C \int_0^T h^{-1}  \|u - C_h u\|^2_\varpi ~\mbox{d}t .
\end{equation}
By combining the inequality
\[
|v_h|_{s,\varpi} \leq C h^{-\frac12} \|v_h\|_\varpi
\]
(that is immediate by \eqref{eq:trace}, \eqref{eq:inverse_ineq} and
\eqref{eq:move_varpi}) with \eqref{eq:clem_stab} we also have
\begin{equation}\label{eq:s_clem_part}
\int_0^T |\pi_h u_h|^2_{s,\varpi} ~\mbox{d}t  \leq C\int_0^T
(h^{-1} \|u - C_h u_h\|^2_{\varpi} +|C_h  u_h|^2_{s,\varpi})  ~\mbox{d}t.
\end{equation}
We conclude as before after applying \eqref{eq:clem_approx_L2} and
\eqref{eq:clem_approx_s} in \eqref{eq:L2_clem_part} and \eqref{eq:s_clem_part}.
\end{proof}
\subsection{Time discretization and weakly imposed boundary
  conditions}
In practice and in the numerical section below of course we need to
include boundary conditions and time discretizations in the above
arguments. Depending on the time-discretization this can be a
challening exercise, but we will here focus on the $\theta$-scheme and the main steps of its analysis using the ideas
above in the case of the backward Euler scheme ($\theta=1$). Boundary conditions
are imposed weakly using the
standard upwind technique known from discontinuous Galerkin methods.
We consider a polygonal domain $\Omega$ and denote its boundary by
$\Gamma:= \partial \Omega$ with outward pointing normal $n$. We decompose $\Gamma$ into an inflow part
\[
\Gamma_- := \{x \in \Gamma: \bfbeta(x) \cdot n<0\}
\]
and an outflow part $\Gamma_+ := \partial \Omega \setminus
\Gamma_-$. The space $V_h$ will here denote the standard finite element space
of continuous piecewise polynomial functions, without boundary
conditions defined on $\mathcal{T}$. We are now interested in the the solution of
\eqref{eq:1st_order} with the additional inflow boundary condition
\[
u = g \mbox{ on } \Gamma_-
\]
where $g \in L^2(0,T;L^2_{\beta \cdot n}(\Gamma-))$ with
$
L^2_{\bfbeta \cdot n}(\Gamma_-) := \{ v:\Gamma_- \mapsto \mathbb{R}: \||\bfbeta
\cdot n|^{\frac12} v\|_{L^2(\Gamma_-)} < \infty\}.
$
We will assume that the $g$, $\Gamma_-$ and $\Gamma_+$ are such that
the exact solution is smooth enough for our purposes.
The timestep $\delta t:= T/N$ for some $N \in \mathbb{N}^+$ will be
assumed satisfy $\delta t \leq C h$ for some $C>0$,
and the discrete solution $u_h := \{u_h^n\}_{n=0}^N$ collects the
finite element approximations on the discrete time levels $t^n = n
\delta t$. The
so-called $\theta$-scheme takes the form: find $u_h^{n} \in V_h$ such
that for $ n = 1,2,3 \dots N$,
\begin{equation}\label{eq:theta}
(\mathcal{L}_\theta^n u_h, v_h)_\Omega + \left< |\bfbeta \cdot n|
  u^{n_\theta}_h,v_h \right>_{\Gamma_-} + s(u_h^{n_\theta},v_h) = (f^{n_\theta},v_h)_\Omega +
\left<|\bfbeta \cdot n| g^{n_\theta},v_h \right>_{\Gamma_-},
\forall v_h \in V_h,
\end{equation}
where  $u_h^{n_\theta} := \theta u_h^n + (1-\theta) u_h^{n-1}$,
$g^{n_\theta} := g(\cdot, t^n+\theta \delta t)$,
$f^{n_\theta}:=f(\cdot, t^n+\theta \delta t)$,
\[
\mathcal{L}_\theta^n u_h := \delta t^{-1} (u_h^n - u_h^{n-1}) + \bfbeta \cdot
\nabla u_h^{n_\theta}, \quad \theta \in [1/2,1]
\]
and $u_h^0 = \pi_h u_0$.
Compared to the time continuous analysis we have two additional points
to study
\begin{enumerate}
\item the time discrete character of the equation
\item the boundary penalty term.
\end{enumerate}
We recall that the theta scheme includes the well-known backward Euler
scheme ($\theta =1$) and the Crank-Nicolson scheme ($\theta=1/2$). A
complete
analysis of the $\theta$ scheme is beyond the scope of the present
paper. To give some insight in the validity of the above arguments in
the fully discrete case we will show the modifications necessary to prove Proposition
\ref{prop:weight_stab} in the time discrete case with weakly imposed
boundary conditions, for $\theta=1$. The Theorem \ref{weigh_error} then follows using
the arguments above and standard truncation error analysis. We will
then show numerically that also the Crank-Nicolson scheme enjoys the
local accuracy property. For further evidence of the local accuracy
property we refer to \cite[Section 5.2 and Fig. 1]{BEF10} for examples
using explicit Runge-Kutta methods and \cite[Section 6]{BG20} for
examples using explicit extrapolated multistep methods.
For the analysis we need the following Lemma the proof of which is
given in the Appendix.
\begin{lemma}\label{lem:technical_tim}
Let $\varpi_n(x) = \varpi(x,t_n)$, where $\varpi$ is a weightfunction
satisfying \eqref{eq:varpi_bound} and $v_h \in V_h$, then for $\delta t$
small enough there holds
%\[
%\sup_{t \in [t_{n-1}, t_n]} \|v_h\|_{\varpi(t)} \leq C K^{-1/2}\|v_h\|_{\varpi_{n}}
%\]
%and
\[
 \|v_h\int_{t_{n-1}}^{t_n} \partial_t
  \varpi ~\mbox{d}t\|_\Omega + \|v_h\left|\int_{t_{n-1}}^{t_n} \int_{t}^{t_n}\partial_t^2
  \varpi^2 ~\mbox{d}s\,\mbox{d}t\right|^{\frac12}\|_\Omega\leq C K^{-1}\delta t^{\frac12} \|v_h\|_{\varpi_{n}}.
\]
\end{lemma}
The following weighted $L^2$-stability estimate is the key ingredient
of the analysis of the fully discrete scheme.
\begin{proposition}
Consider the scheme \eqref{eq:theta} with $\theta=1$, then assuming
$\delta t<1$ small enough there holds, with $w_h^n =
  \pi_h \varpi^2 v_h^n$,
\begin{align*}
\|v_h^{N}\|_{\varpi_{N}}^2 &+ \sum_{n=1}^N \|v_h^n -
v_h^{n-1}\|_{\varpi_{n}}^2+\delta t \sum_{n=1}^N(\||\bfbeta \cdot n|^{\frac12}
  v_h^n \varpi_n\|_{\Gamma}^2 + \gamma |v_h^n|_{s,\varpi_n}^2) \\
& \leq
  C_K (\|v_h^{0}\|_{\varpi_{0}}^2+\delta t \sum_{n=1}^N ( (\mathcal{L}_\theta^n v_h, w_h^n)_\Omega  + \left< |\bfbeta \cdot n|
  v^{n}_h,w^n_h \right>_{\Gamma_-} +
  \gamma s(v_h^{n},w^n_h))).
\end{align*}
The constant $C_K$ grows exponentially in time with exponential coefficient $1/K^2$.
\end{proposition}
\begin{proof}
First we observe that using standard partial integration and $\nabla
\cdot \bfbeta = 0$ we have
\[
(\bfbeta \cdot \nabla v_h, \varpi^2 v_h)_\Omega +\left< |\bfbeta \cdot n|
  v_h,\varpi^2 v_h \right>_{\Gamma_-}= - (\bfbeta \cdot \nabla v_h, \varpi^2 v_h)_\Omega- (v_h,(\bfbeta \cdot \nabla \varpi^2)
v_h)_\Omega + \left< |\bfbeta \cdot n| v_h, \varpi^2 v_h\right>_{\Gamma_+}.
\]
As a consequence
\[
(\bfbeta \cdot \nabla v_h, \varpi^2 v_h)_\Omega +\left< |\bfbeta \cdot n|
  v_h,\varpi^2 v_h \right>_{\Gamma_-}= - \frac12 (v_h,(\bfbeta \cdot \nabla \varpi^2)
v_h)_\Omega + \frac12 \left< |\bfbeta \cdot n| v_h, \varpi^2 v_h\right>_{\Gamma}.
\]
We also have 
\[
(v_h^n - v_h^{n-1},\varpi^2_{n} v^n_h)_\Omega =
\frac12 \|v_h^{n}\|_{\varpi_{n}}^2 + \frac12 \|v_h^n -
v_h^{n-1}\|_{\varpi_{n}}^2- \frac12 \|v_h^{n-1}\|_{\varpi_{n}}^2.
\]
It follows that 
\begin{align*}
\delta t \sum_{n=1}^N ( (\mathcal{L}_\theta^n v_h, \varpi^2_{n}
  v^n_h)_\Omega & + \left< |\bfbeta \cdot n|
  v^{n}_h,\varpi^2_{n} v^n_h \right>_{\Gamma_-} +
  \gamma s(v_h^{n},\varpi^2_{n} v^n_h) ) \\
& = \frac12 \|v_h^{N}\|_{\varpi_{N}}^2 +  \frac12 \sum_{n=1}^N (\|v_h^n -
v_h^{n-1}\|_{\varpi_{n}}^2-
  ((v_h^{n-1})^2,\varpi_{n}^2-\varpi_{n-1}^2)_\Omega) - \frac12
  \|v_h^{0}\|_{\varpi_{0}}^2\\
& - \frac12 \delta t \sum_{n=1}^N ((v_h^n)^2,\bfbeta \cdot \nabla \varpi^2_n)_\Omega + \frac12 \delta t \sum_{n=1}^N(\||\bfbeta \cdot n|^{\frac12}
  v_h^n \varpi_n,\|_{\Gamma}^2 + 2 \gamma s(v_h^n,\varpi_n^2 v_h^n)).
\end{align*}
Identifying the terms in the right hand side that do not have a sign
we see that we need to control
\[
 \sum_{n=1}^N
 ((v_h^{n-1})^2,\varpi_{n}^2-\varpi_{n-1}^2)_\Omega+\delta t ((v_h^n)^2,\bfbeta \cdot \nabla \varpi^2_n)_\Omega).
\]
We rewrite the first term 
\[
((v_h^{n-1})^2, (\varpi_{n}^2-\varpi_{n-1}^2) )_\Omega = ((v_h^{n-1})^2 -
(v_h^{n})^2,\varpi_{n}^2-\varpi_{n-1}^2)_\Omega + ((v_h^{n})^2, (\varpi_{n}^2-\varpi_{n-1}^2))_\Omega.
\]
For the first term on the right hand side we develop $a^2-b^2 =
(a+b)(a-b)$ and apply Cauchy-Schwarz inequality and the
arithmetic-geometric inequality, followed by Lemma
\ref{lem:technical_tim} and the inequality \eqref{eq:move_varpi} to
obtain the bound
\begin{align*}
 ((v_h^{n-1})^2 -
(v_h^{n})^2,\varpi_{n}^2-\varpi_{n-1}^2)_\Omega & =
                                                ((v_h^{n-1}+v_h^{n})(v_h^{n-1}-v_h^{n}),\varpi_{n}^2-\varpi_{n-1}^2)_\Omega\\
& =
  ((v_h^{n-1}+v_h^{n})(v_h^{n-1}-v_h^{n}),(\varpi_{n}+\varpi_{n-1})\int_{t_{n-1}}^{t_n} \partial_t
  \varpi(\cdot,t) ~\mbox{d}t)_\Omega\\
&\ge - \epsilon^{-1} \|(v_h^n+v_h^{n-1}) \int_{t_{n-1}}^{t_n} \partial_t
  \varpi(\cdot,t) ~\mbox{d}t\|_\Omega^2 - \frac{\epsilon}{2}
  ((v_h^n-v_h^{n-1})^2, \varpi_{n}^2+\varpi_{n-1}^2)_\Omega \\
&\ge - C K^{-2} \epsilon^{-1} \delta t (\|v_h^n\|_{\varpi_n}^2+
  \|v_h^{n-1}\|_{\varpi_{n-1}}^2)-  \frac{C \epsilon}{2} \|v_h^n-v_h^{n-1}\|_{\varpi_{n}}^2.
\end{align*}
Considering the remaining terms, using the relation $\mathcal{L}
\varpi^2 = 0$, and applying once again Lemma
\ref{lem:technical_tim}, yields the bound
\begin{align*}
((v_h^{n-1})^2,\varpi_{n}^2-\varpi_{n-1}^2)_\Omega +  \delta t
  ((v_h^n)^2,\bfbeta \cdot \nabla \varpi^2_n)_\Omega & =
                                                     ((v_h^n)^2,\int_{t_{n-1}}^{t_n} \partial_t
                                                     \varpi^2
                                                     ~\mbox{d}t-
                                                     \delta t\partial_t \varpi^2_n)_\Omega \\
& = ((v_h^n)^2,\int_{t_{n-1}}^{t_n} \int_t^{t_n} \partial_{tt} \varpi^2
  ~\mbox{d}s ~\mbox{d}t)_\Omega\\
& \ge  - \delta t C/K^2 \|v_h^n\|_{\varpi_n}^2.
\end{align*}
Taking $\epsilon$  sufficiently small so that $C
\epsilon/2\leq 1/4$ it
follows that
\begin{align*}
\|v_h^{N}\|_{\varpi_{N}}^2 &+ \sum_{n=1}^N (\|v_h^n -
v_h^{n-1}\|_{\varpi_{n}}^2+\delta t \sum_{n=1}^N(\||\bfbeta \cdot n|^{\frac12}
  v_h^n \varpi_n\|_{\Gamma}^2 + \gamma |v_h^n|_{s,\varpi_n}^2) \\
& \leq C (
  \|v_h^{0}\|_{\varpi_{0}}^2+\delta t \sum_{n=1}^N ( (\mathcal{L}_\theta^n v_h, \varpi^2_{n}
  v^n_h)_\Omega  + \left< |\bfbeta \cdot n|
  v^{n}_h,\varpi^2_{n} v^n_h \right>_{\Gamma_-} +
  \gamma s(v_h^{n},\varpi^2_{n} v^n_h) +\delta t C K^{-2} \|v_h^n\|_{\varpi_n}^2)).
\end{align*}
Proceeding as before we add and subtract $w_h^n := \pi_h (\varpi^2_{n}
v^n_h)$ in the right slot of the bilinear forms of the right hand side 
\begin{align*}
\|v_h^{N}\|_{\varpi_{N}}^2 &+ \sum_{n=1}^N \|v_h^n -
v_h^{n-1}\|_{\varpi_{n}}^2+\delta t \sum_{n=1}^N(\||\bfbeta \cdot n|^{\frac12}
  v_h^n \varpi_n\|_{\Gamma}^2 + \gamma |v_h^n|_{s,\varpi_n}^2) \\
& \leq
  C(\|v_h^{0}\|_{\varpi_{0}}^2 +\delta t \sum_{n=1}^N ( (\mathcal{L}_\theta^n v_h, w_h)_\Omega  + \left< |\bfbeta \cdot n|
  v^{n}_h,w_h \right>_{\Gamma_-} +
  s(v_h^{n},w_h) +\delta t C \|v_h^n\|_{\varpi_n}^2)\\
& +\delta t \sum_{n=1}^N ( (\mathcal{L}_\theta^n v_h, \varpi^2_{n} 
  v^n_h - w_h)_\Omega  + \left< |\bfbeta \cdot n|
  v^{n}_h,\varpi^2_{n} v^n_h - w_h \right>_{\Gamma_-} +
    \gamma s(v_h^{n},\varpi^2_{n} v^n_h - w_h))).
\end{align*}
Only the term introduced for the weak imposition of boundary
conditions differs from the time-continuous analysis. For this term we
observe that
\[
 \left< |\bfbeta \cdot n|
  v^{n}_h,\varpi^2_{n} v^n_h - w_h \right>_{\Gamma_-} \ge -\epsilon \||\bfbeta \cdot n|^{\frac12}
  v_h^n \varpi_n\|_{\Gamma}^2   - \frac{\beta_\infty}{4 \epsilon} \|\varpi_{n}^{-1}
(\varpi^2_{n} v^n_h - \pi_h \varpi^2_{n} v^n_h)\|_{\Gamma_-}^2.
\]
For the second term on the right hand side we have the bound
\[
 \|\varpi_{n}^{-1}
(\varpi^2_{n} v^n_h - \pi_h \varpi^2_{n} v^n_h)\|_{\Gamma_-}^2 \leq C/K^2 \|v^n_h \|_{\varpi}^2.
\]
This follows by applying the trace inequality \eqref{eq:trace}, the
properties of $\varpi$ and the inequality \eqref{super1}.
Proceeding as in the time-continuous case we then obtain the bound
\begin{align*}
\|v_h^{N}\|_{\varpi_{N}}^2 &+ \sum_{n=1}^N \|v_h^n -
v_h^{n-1}\|_{\varpi_{n}}^2+\delta t \sum_{n=1}^N(\||\bfbeta \cdot n|^{\frac12}
  v_h^n \varpi_n\|_{\Gamma}^2 + \gamma |v_h^n|_{s,\varpi_n}^2) \\
& \leq C(
  \|v_h^{0}\|_{\varpi_{0}}^2+ \delta t \sum_{n=1}^N ( (\mathcal{L}_\theta^n v_h, w_h)_\Omega  + \left< |\bfbeta \cdot n|
  v^{n}_h,w_h \right>_{\Gamma_-} +
  \gamma s(v_h^{n},w_h) +\delta t  K^{-2} \|v_h^n\|_{\varpi_n}^2)).
\end{align*}
Choosing $\delta t$ sufficiently small the term $\delta t C
\|v_h^N\|_{\varpi_N}^2$ in the right hand side can be absorbed in the
left hand side and we conclude by an application of the discrete
Gronwall's inequality.
\end{proof}
\begin{remark}
A consequence of the previous analysis is that the proposed
method can be used in the context of problems, where the boundary or
initial data
is unknown or partially known. Assume for example that
$g$ is unknown and replaced by zero. Then, since the effect of the
erroneous boundary condition is damped exponentially for
non-characteristic directions, the solution can still be approximated
with good accuracy in subsets $\Omega_0$ whose domain of dependence is
sufficiently far from the boundary. Similarly if the initial data is
unknown in some parts of the domain, the solution will still remain
accurate in subdomains where the initial data in the domain of
dependence is known. This result is a time-dependent
analogue to the analysis of \cite{BNO20}.
\end{remark}

\section{Numerical examples}\label{sec:numerics}
All numerical examples were produced using the package FreeFEM++
\cite{He12}.
The method \eqref{eq:theta} is considered with $\theta=1/2$,
corresponding to the second order Crank-Nicolson scheme. This choice
was made to minimize the perturbation of the global energy estimate by
the time-discretization. The
consistent mass matrix is used and exact quadrature is applied to all
the forms.
We first consider transport in the disc $\Omega:= \{(x,y) \in \mathbb{R}^2 :
x^2+y^2  < 1\}$ under the velocity field $\bfbeta =
(y,-x)$. Approximations are computed on a series of unstructured
meshes. We set
$f=0$ and consider two different functions $u_0$ as initial data.
One is smooth
\begin{equation}\label{eq:u0_smooth}
u_0 = e^{-30((x-0.5)^2+y^2)}
\end{equation}
and one is rough 
\[
\tilde u_0 = \left\{\begin{array}{l}
1 \quad \sqrt{(x+0.5)^2 +y^2} <0.2\\
0 \quad \mbox{otherwise}
\end{array} \right. .
\]
The velocity field simply turns the disc with the initial data and 
one full turn is computed so that the final solution should be equal to the
inital data. Two numerical experiments are considered where the
solution is approximated for
the initial data $u_0$ and $u_0 + \tilde u_0$.
\begin{figure}[t]
\centering
\hspace{-0.5cm}
\includegraphics[width=0.3\linewidth, angle=90]{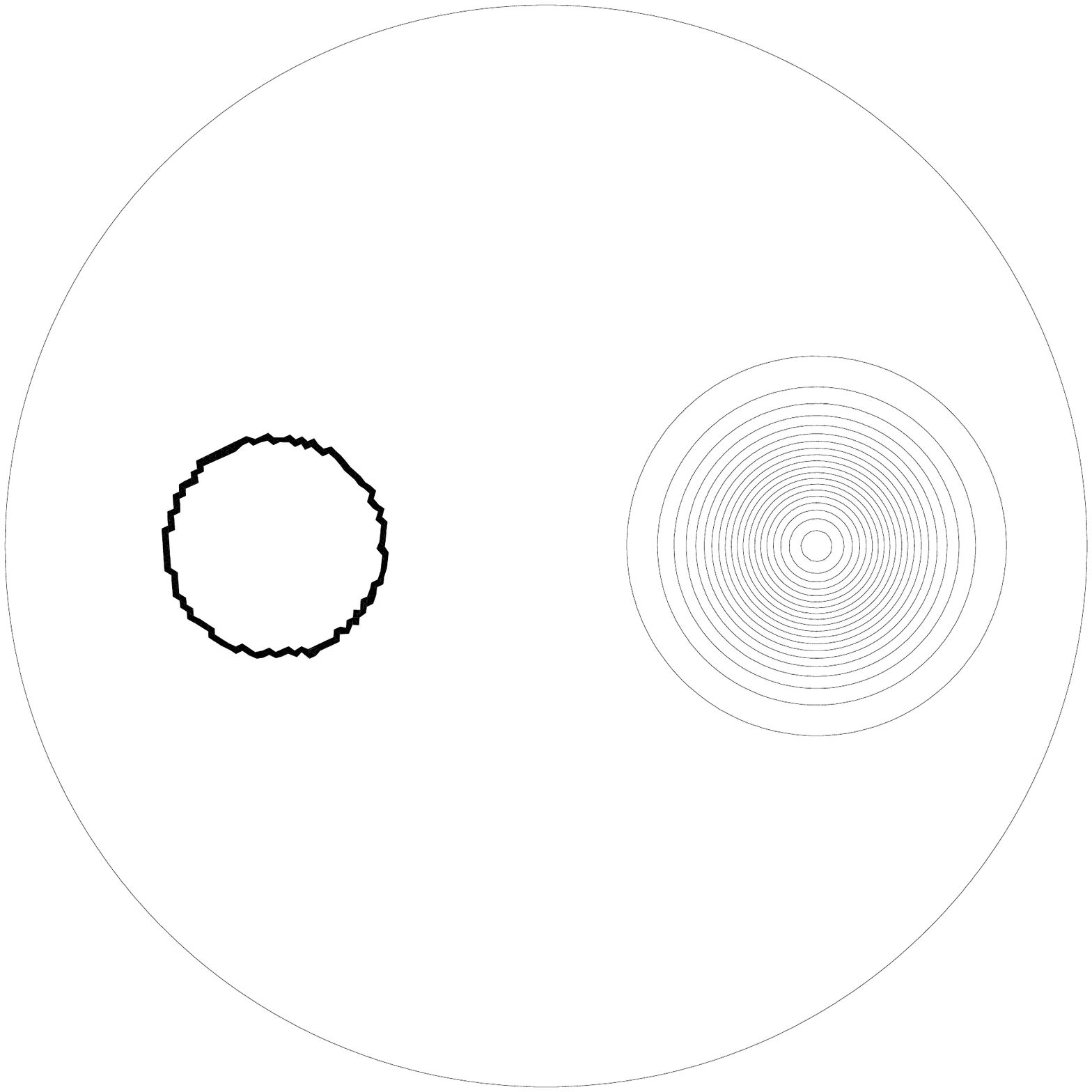} \hspace{-2cm}
\includegraphics[width=0.3\linewidth, angle=90]{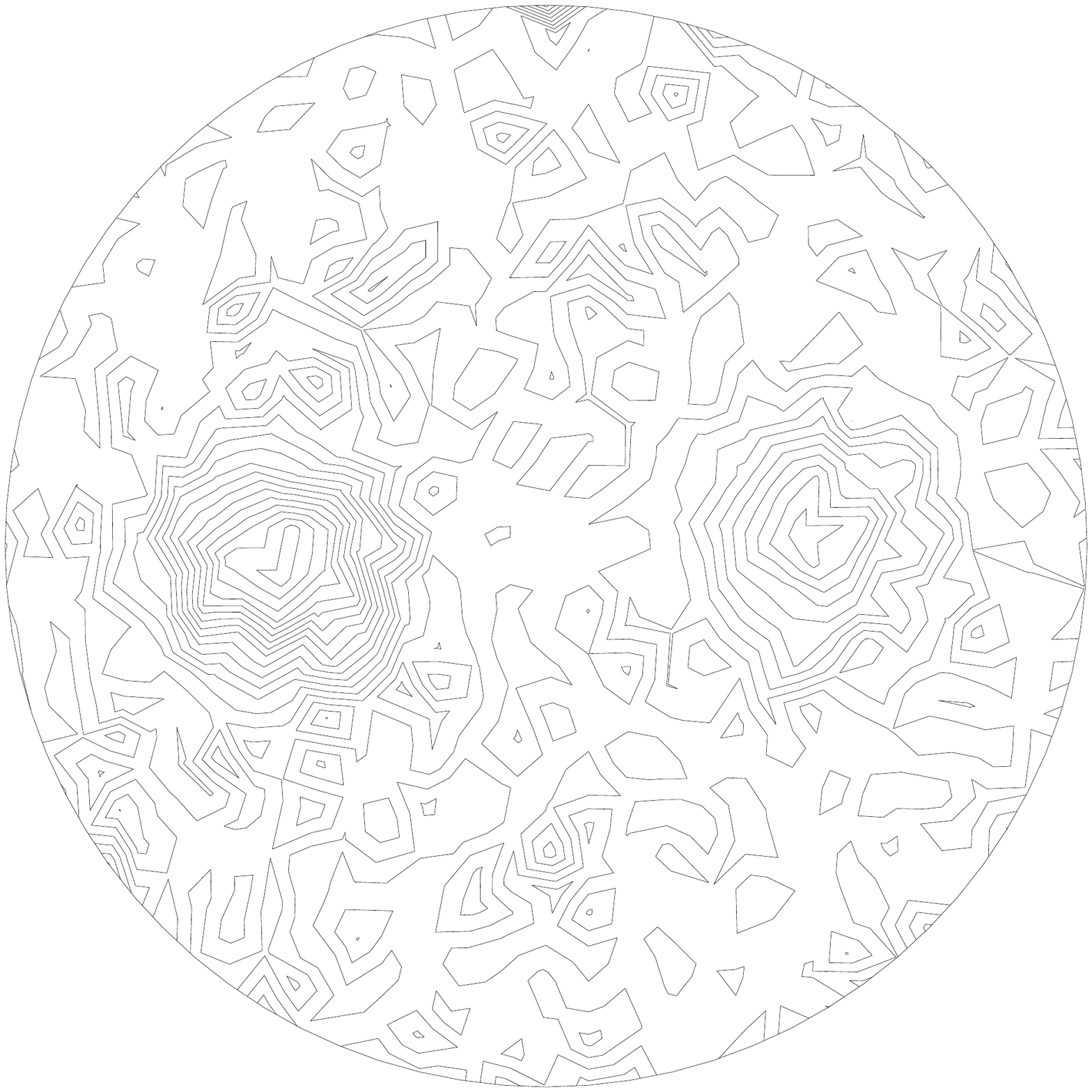}\hspace{-2cm}
\includegraphics[width=0.3\linewidth, angle=90]{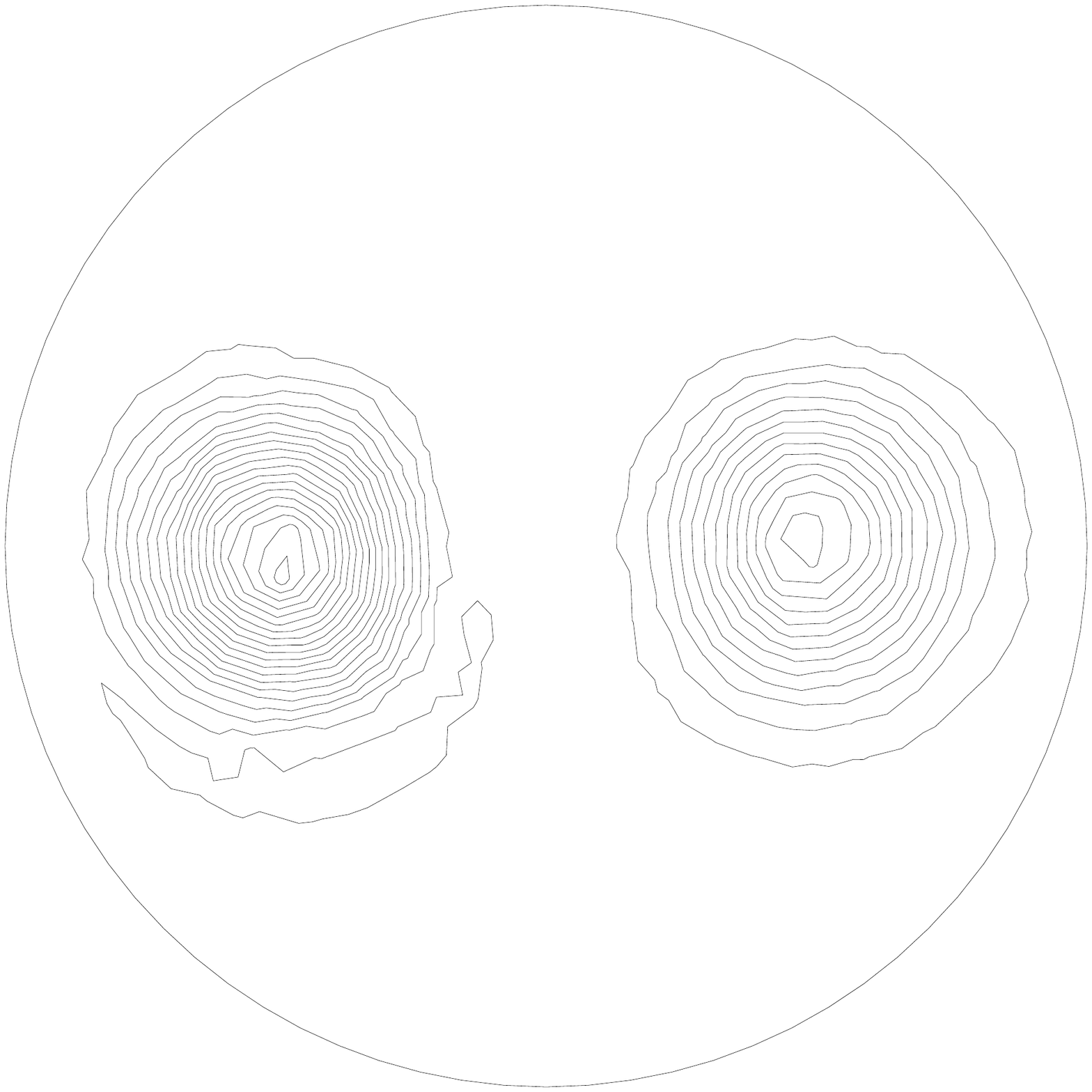}
\caption{From left to right: rough initial data on fine mesh $u_0 + \tilde u_0$, unstabilized
  solution, stabilized solution ($nele=80$, one turn).}
\label{fig:data1}
\end{figure}
We report the global error in the material derivative over the space
time domain, the global
$L^2$-norm of the error at the final time, and in the case where both the rough and the
smooth initial data are combined, the error obtained in the
smooth part, i.e. the $L^2$-norm over $\{(x,y) \in \Omega: x>0\}$. The discretization parameters
for piecewise affine ($P_1$ below) approximation have been chosen as $dt = \tfrac12 h =
\pi/nele$, where $nele$ is the number of cell faces on the disc
perimeter. For piecewise quadratic ($P_2$ below) approximation $h = 2
\pi/nele$ and $dt = \tfrac12 h^{\frac32}$, to make the error of the time
and space discretization similar.  In the left panel of figure
\ref{fig:data1} the smooth and rough initial data, interpolated on a very fine mesh, are presented. In the
middle panel the solution after one turn without stabilization and in
the right panel the solution after one turn with stabilization for $P_1$,
on the mesh resolution $nele =80$ are reported. We see
that the sharp layers are smeared on this coarse mesh when
the stabilized method is used, but contrary to the unstabilized case
the smooth part of the solution is accurately captured. 

 In figure \ref{fig:conv_smooth1} the convergence of stabilized and unstabilized methods with
$P_1$ and $P_2$ elements are compared for the smooth initial data. We
observe that when the solution is globally smooth both methods perform
well in the $L^2$-norm. Nevertheless, the improvement of the convergence rate for the
stabilized method is clearly visible for both approximation spaces,
both in the $L^2$-error and in the material derivative. The results when part of the solution is
rough (initial data from figure \ref{fig:data1}, left plot) are
reported in figure \ref{fig:conv_smooth2}. Note that both methods have similar global error
in the $L^2$-norm. The stabilized method on
the other hand still has optimal convergence in the part where the
solution is smooth, in accordance with the theory of section
\ref{sec:weight}. Its material derivative is also more stable under refinement.
The unstabilized method has equally poor convergence in the smooth and
in the rough part of the solution.
\begin{figure}[t]
\centering
\includegraphics[width=0.45\linewidth]{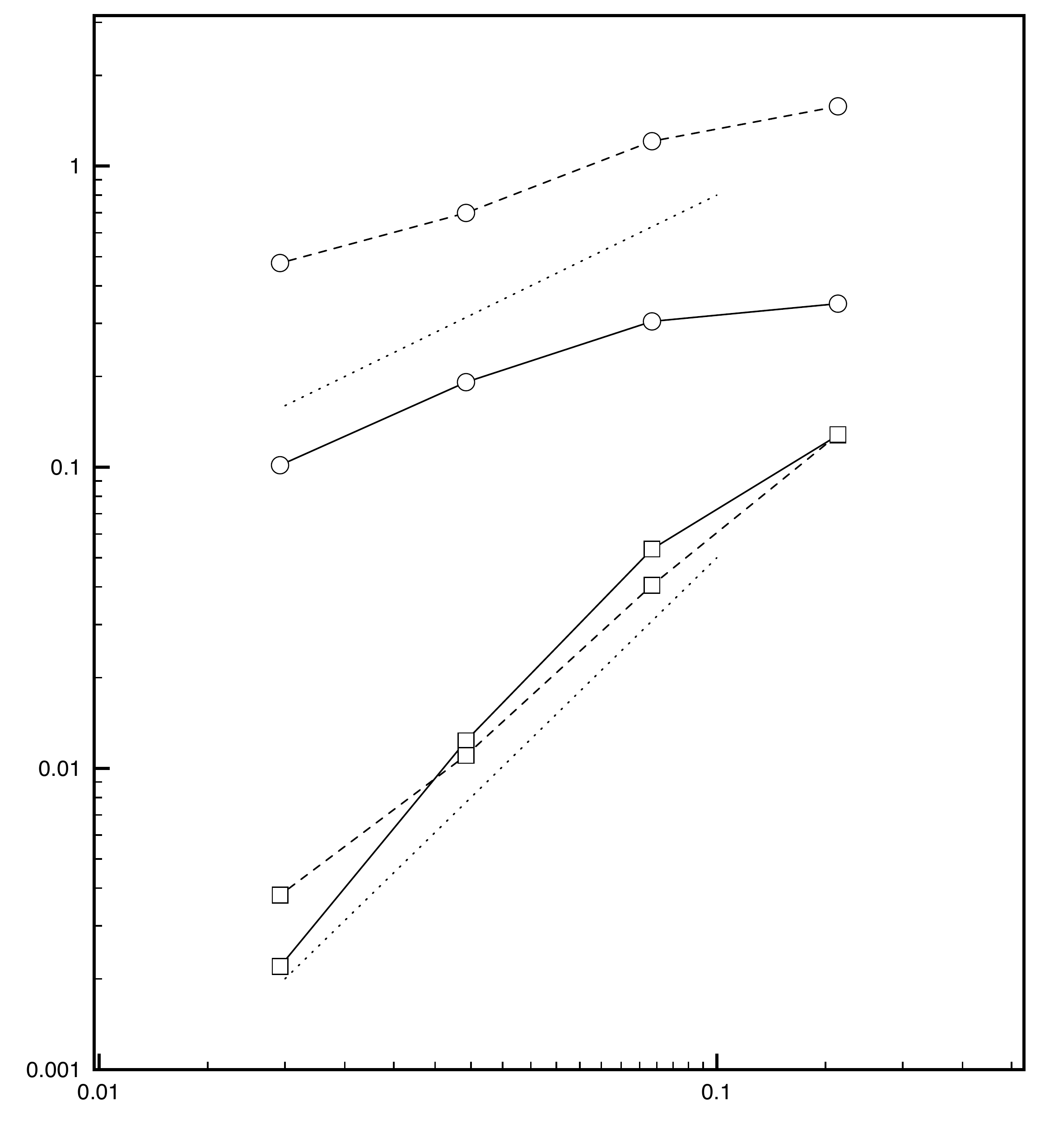}
\includegraphics[width=0.45\linewidth]{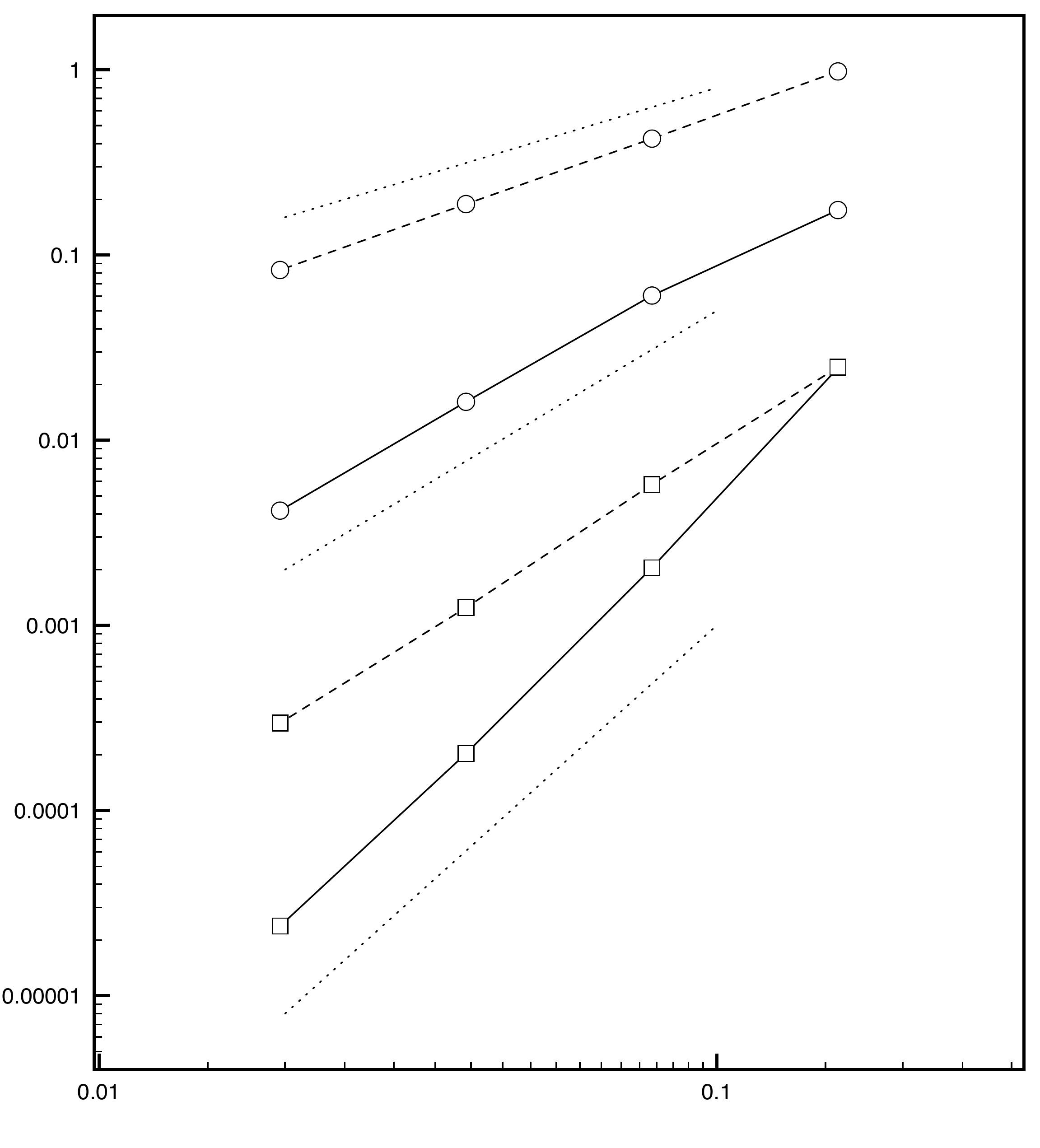}
\caption{Comparison of errors plotted against mesh size $h$ for stabilized (full line) and unstabilized (dashed
  line) methods with $P_1$ (left) and $P_2$ (right) approximation.
  Globally smooth initial data (equation \eqref{eq:u0_smooth}). The
  space time 
  error in
  material derivative has circle markers. The final time
 global $L^2$-error has square markers. The dotted reference lines
 have slope $1,2$ from top to bottom in the left graphic and $1,2,3$ from top to bottom in the right graphic.}
\label{fig:conv_smooth1}
\end{figure}
\begin{figure}[t]
\centering
\includegraphics[width=0.45\linewidth]{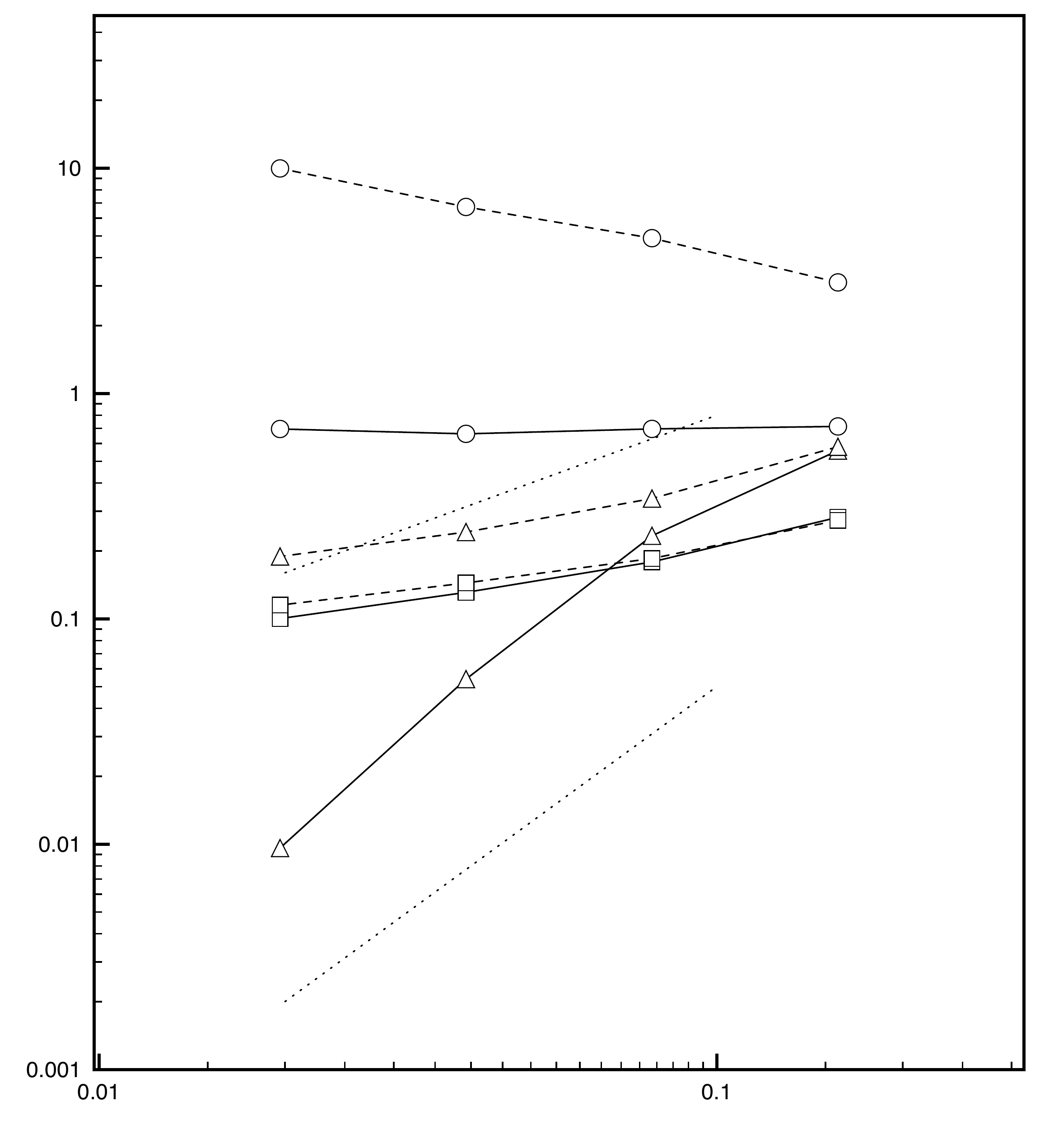}
\includegraphics[width=0.45\linewidth]{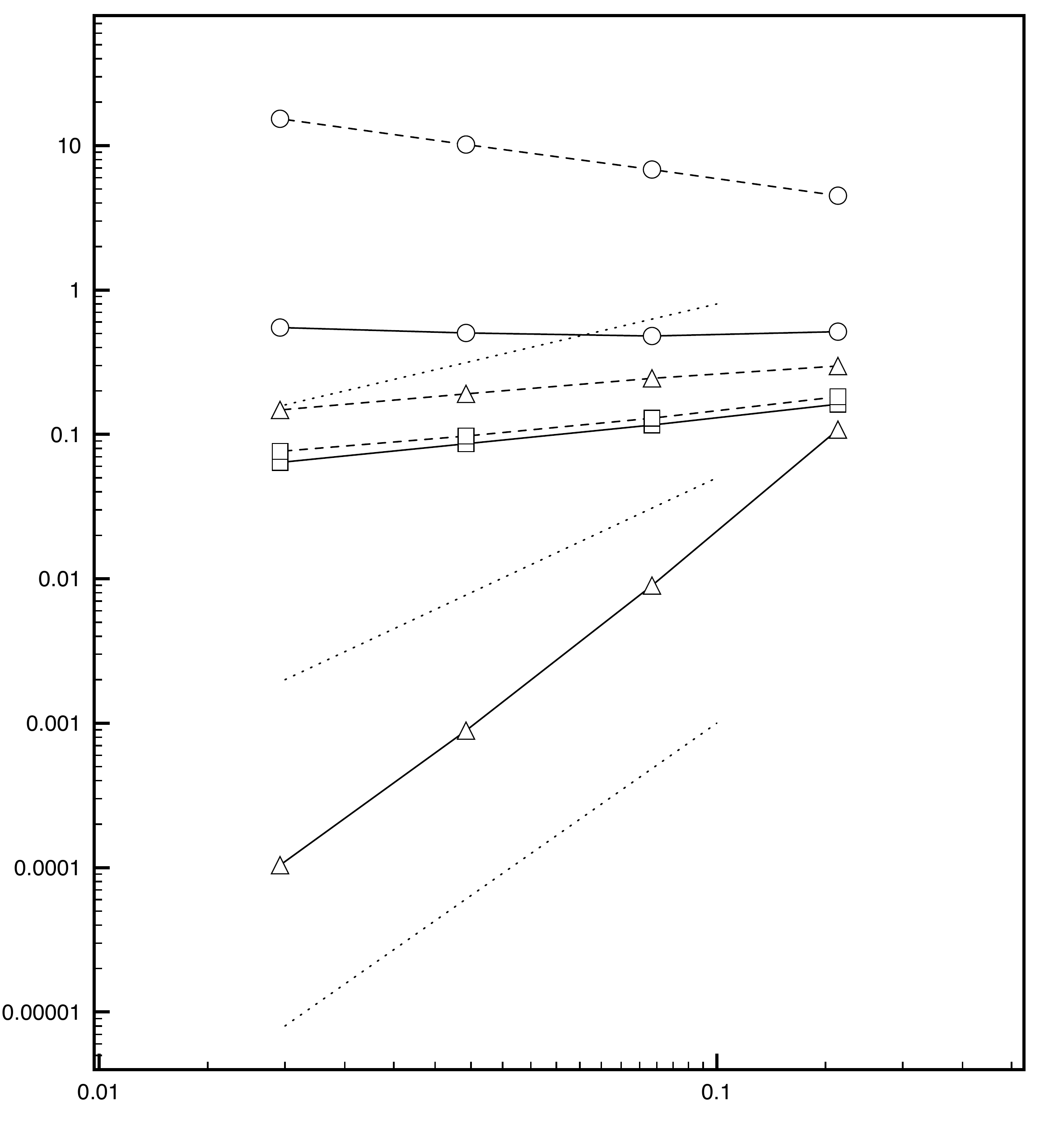}
\caption{Comparison of errors plotted against mesh size $h$ for stabilized (full line) and unstabilized (dashed
  line) methods with $P_1$ (left) and $P_2$ (right). Initial data
  from figure \ref{fig:data1} (left plot). The space time 
  error in
  material derivative has circle markers. The final time 
 global $L^2$-error has square markers and the final time local $L^2$-error has triangle markers. The dotted reference lines
 have slope $1,2$ from top to bottom in the left graphic and $1,2,3$ from top to bottom in the right graphic.}
\label{fig:conv_smooth2}
\end{figure}
\subsection{An example with inflow and outflow and weakly imposed
  boundary conditions}
Here we consider transport in the unit square with $\bfbeta =
(1,0)^T$. We use a structured mesh with $nele$ cell faces on the side of
the square. The initial data consists of a cylinder of radius $r=0.2$
centered in the middle of the square and a Gaussian centered on the
left boundary (See figure \ref{fig:data2}, left plot). The exact
shapes are the same as those of the previous example. The solution is
approximated over the interval $(0,1]$ so
that the cylinder leaves the domain at $t=0.7$ and at $t=1$ the
Gaussian is centered on the right boundary. The time dependent inflow boundary
condition $u =g$ on $\Gamma_-$ is imposed weakly as described in
\eqref{eq:theta} ($g$ is chosen as the trace of the known exact solution). In figure
\ref{fig:data2}, the final time approximation is reported in the middle
plot without
stabilization and the in right plot with stabilization. Observe that from
$t=0.7$ the solution is smooth. 
Nevertheless the unstabilized Galerkin
method fails to produce an accurate approximation of the smooth final
time solution. Spurious oscillations from the discontinuity have spread over the whole
computational domain and remain also when the rough part of the
solution has left. The convergence of the $L^2$-error at final times for
the stabilized and unstabilized  approaches is shown in figure \ref{fig:conv_tube}
($h=1/nele$, $nele = 40,80,160,320$). We see
that for the stabilized method both the $P_1$ and $P_2$ approximations have
optimal convergence to the smooth solution. The unstabilized method
converges approximately as $O(h^{\frac12})$ in both cases and its
material derivative diverges.
\begin{figure}[t]
\centering
\hspace{-0.5cm}
\includegraphics[width=0.3\linewidth, angle=90]{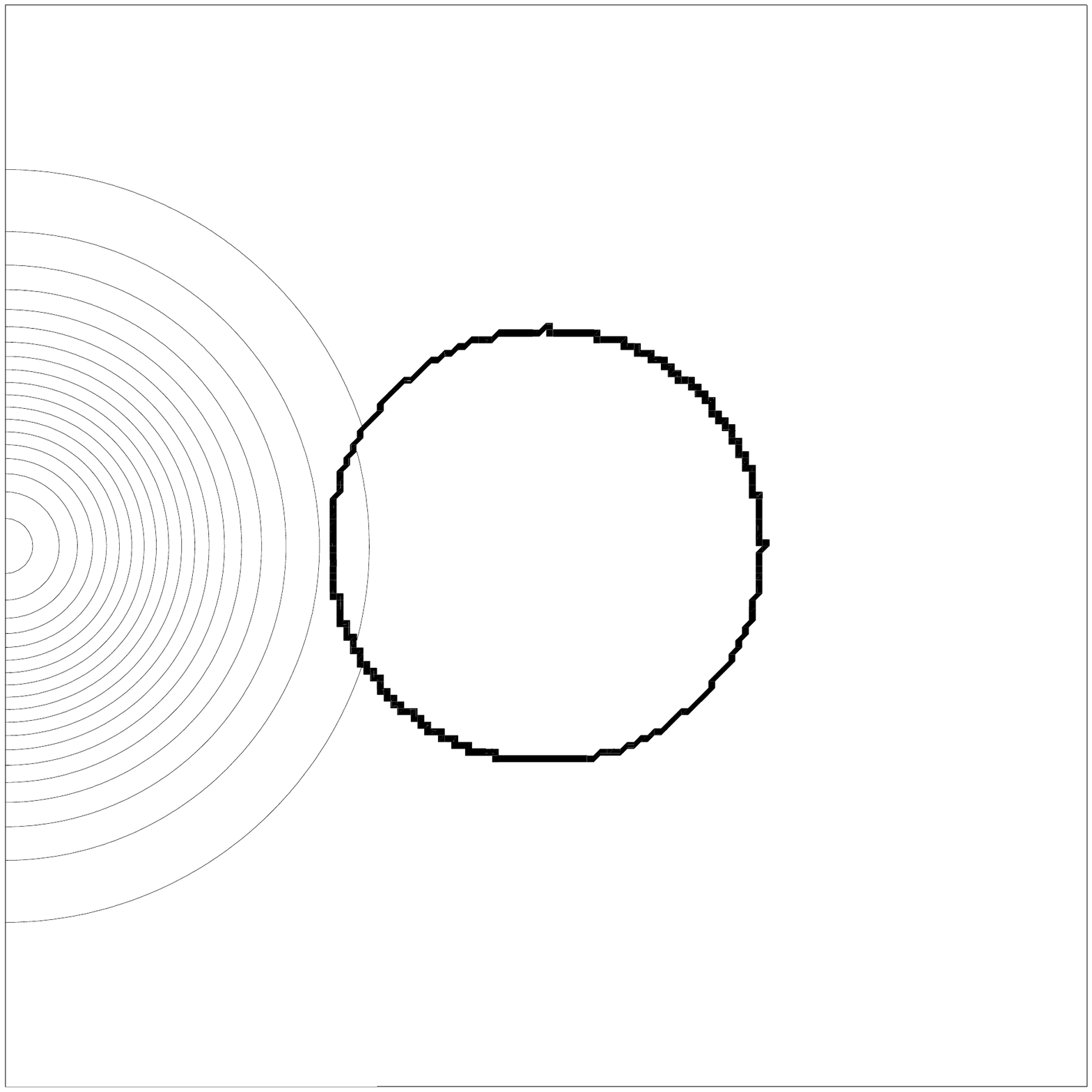} \hspace{-2cm}
\includegraphics[width=0.3\linewidth, angle=90]{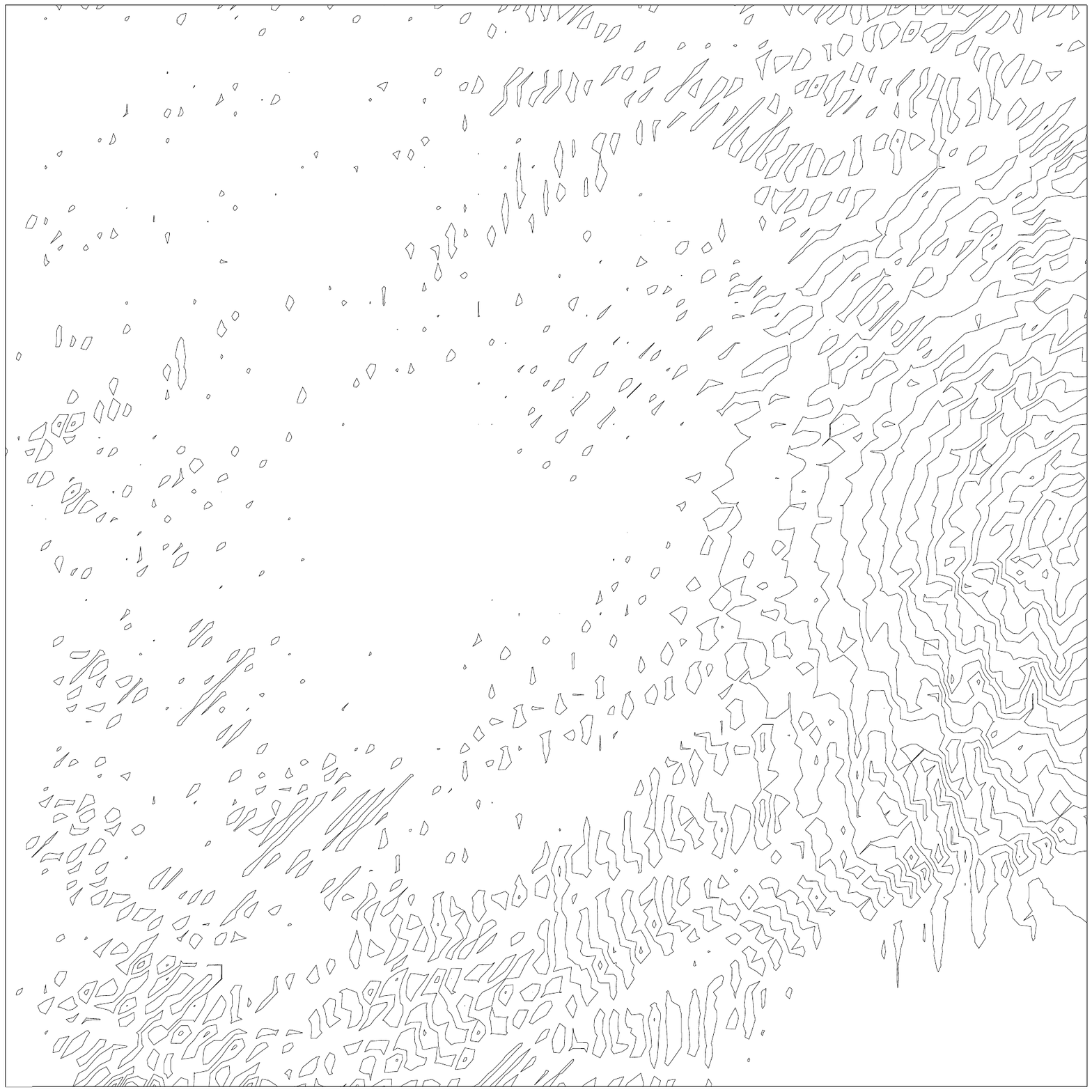}\hspace{-2cm}
\includegraphics[width=0.3\linewidth, angle=90]{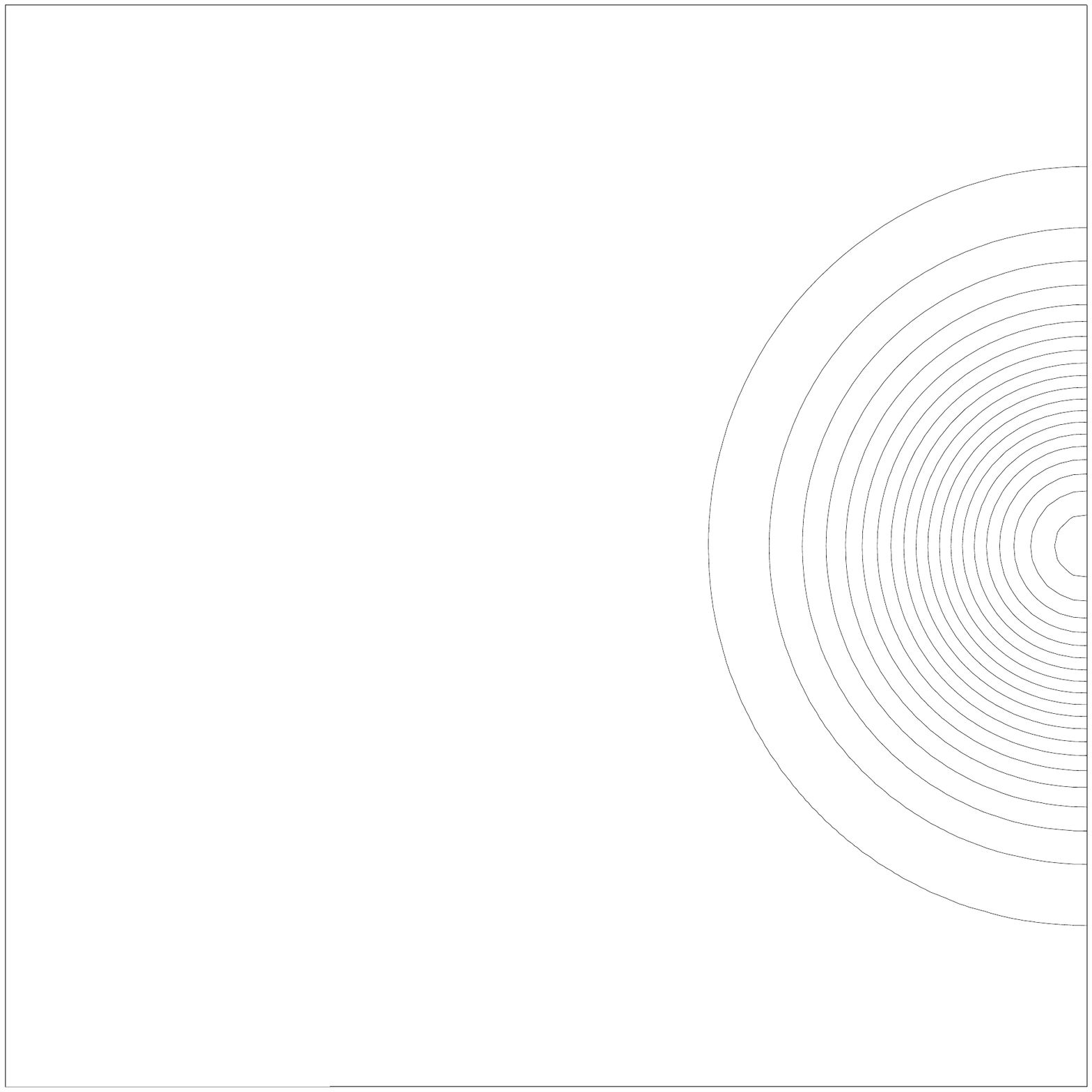}
\caption{From left to right: initial data on fine mesh, unstabilized
  solution, stabilized solution ($nele=80$, final time $t=1$).}
\label{fig:data2}
\end{figure}
\begin{figure}[t]
\centering
\includegraphics[width=0.45\linewidth]{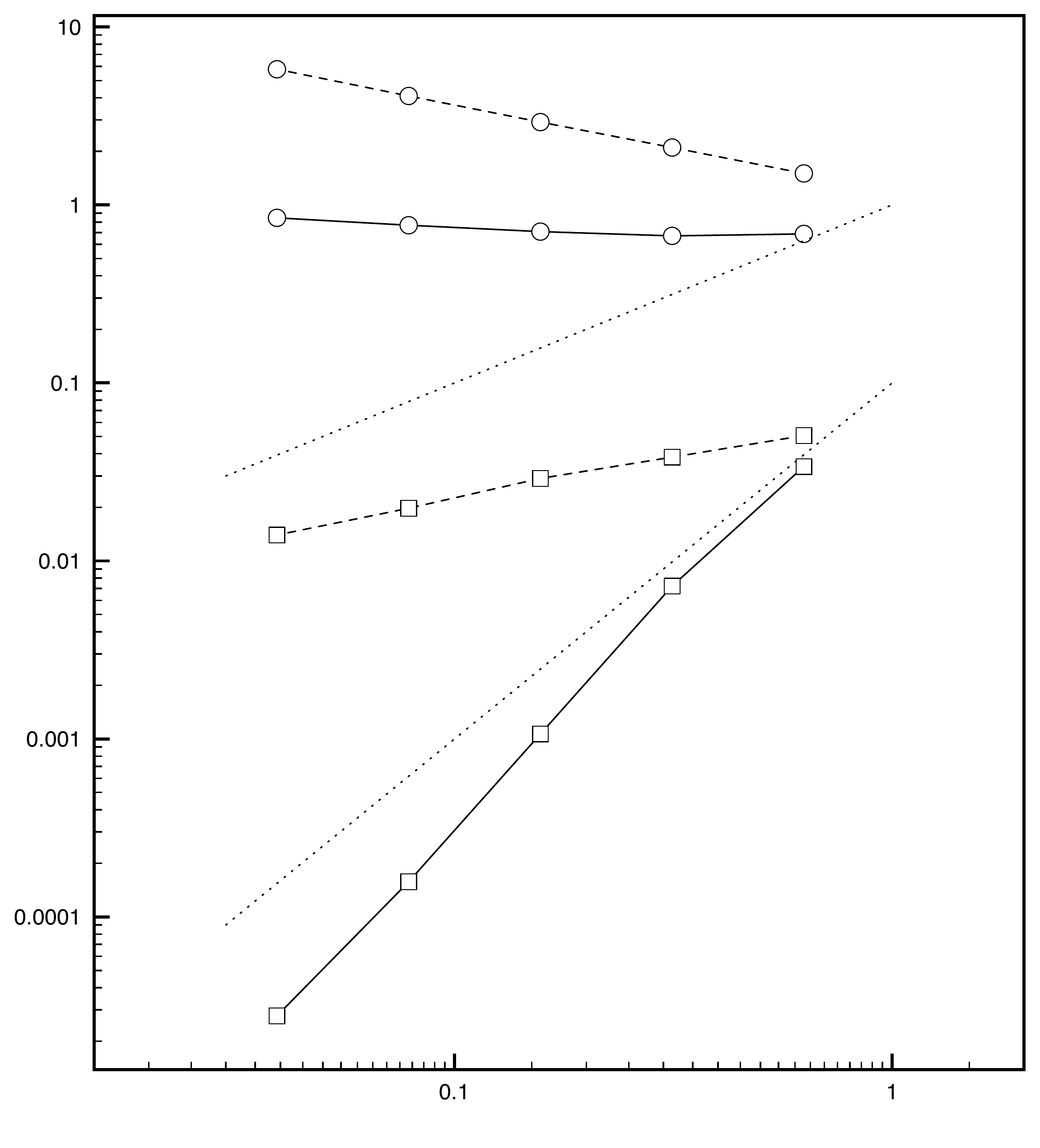}
\includegraphics[width=0.45\linewidth]{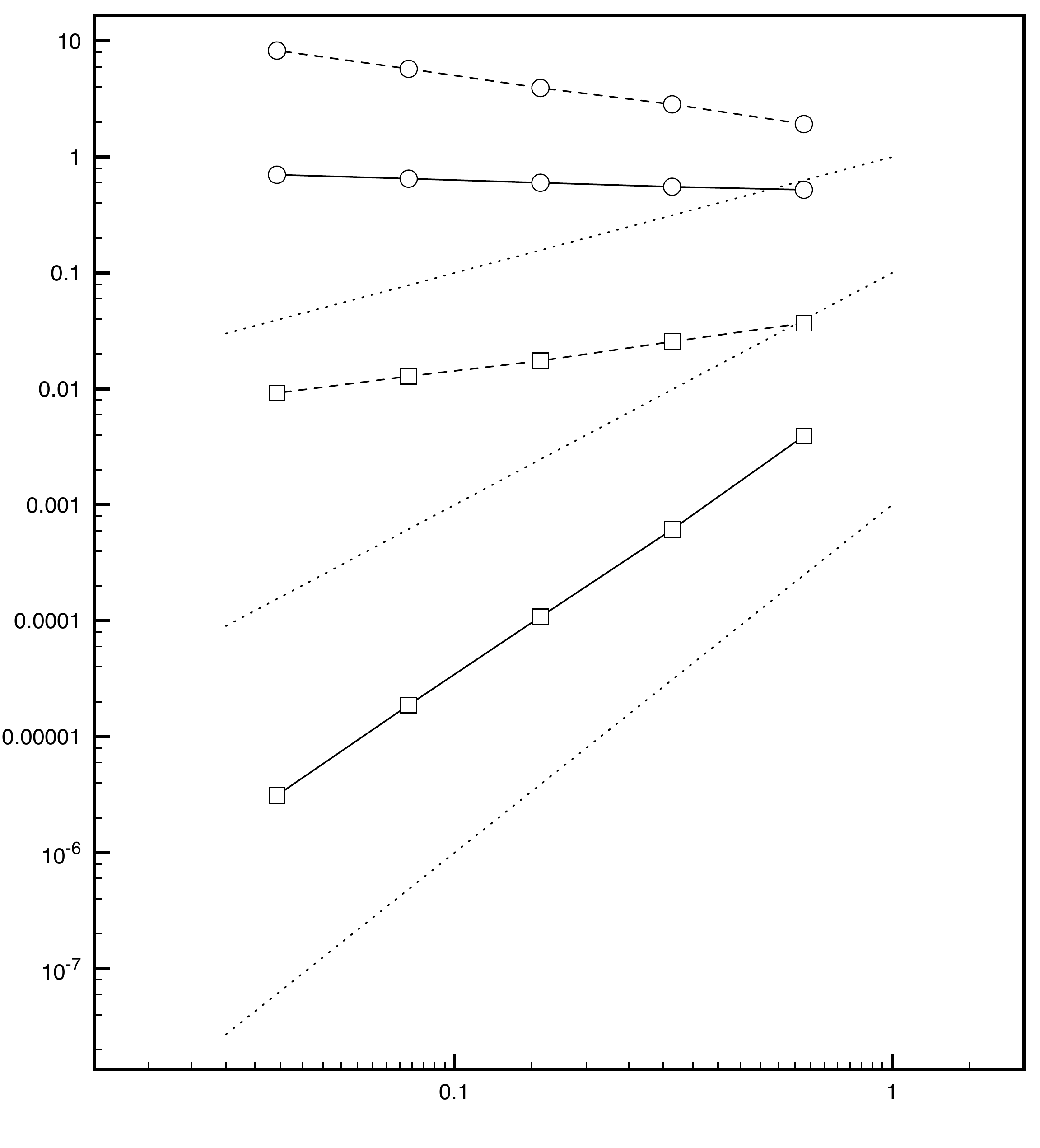}
\caption{Comparison of errors plotted against mesh size $h$ for stabilized (full line) and unstabilized (dashed
  line) methods with $P_1$ (left) and $P_2$ (right).  Initial data
  from figure \ref{fig:data2} (left plot). The space time
  error in
  material derivative has circle markers. The final time 
 global $L^2$-error has square marker. The dotted reference lines
 have slope $1,2$ from top to bottom in the left graphic and $1,2,3$ from top to bottom in the right graphic.}
\label{fig:conv_tube}
\end{figure}
\subsection{Long term stability}
To see the effect of perturbations on the solution for long time we
revisit the computational example of the previous section, but extend
the time interval to $(0,3)$. The cylinder leaves the domain at
$t=0.7$ and at the final time the solution is very small.
One would then expect the error of the method to go to zero with machine
precision, since the solution to approximate is very close to the trivial zero
solution. In figure \ref{fig:longtime} the global $L^2$-norm is reported, for two
consecutive meshes ($nele=40$ and $nele = 80$) and both the stabilized
(full line) and the unstabilized (dashed line) methods. In the
stabilized case the improvement of the approximation at $t=0.7$, when
the cylinder leaves the domain, is clearly visible and the solution
also improves as the Gaussian is evacuated. We see convergence to zero
at machine precision of the error and also convergence under mesh
refinement. In the unstabilized case the change at time $t=0.7$ is
barely visible, the error decreases only very slowly in time and not noticeably
under mesh refinement. Similarly as in the previous example, we conclude that the standard Galerkin method
with weakly imposed boundary conditions in our
  simulations fails to evacuate the high
frequency perturbations produced by the discontinuous initial data on
the two meshes considered.
%\textcolor{red}{This is most likely due to the fact that the standard Galerkin
%method does not transport the high frequency content of the approximate
%solution iwith the right velocity}. 
%Indeed it is not
%even transported in the right direction, so that spurious
%oscillations from the discontinuity spread over the whole
%computational domain as can be seen in the middle graphic of figure \ref{fig:data2}. 
%In particular this shows that
%weak imposition of boundary conditions is not sufficient to make the
%continuous finite element method robust for first order transport
%equations contrary to
%the claim of \cite{ANO20}.
\begin{figure}[t]
\centering
\includegraphics[width=0.45\linewidth]{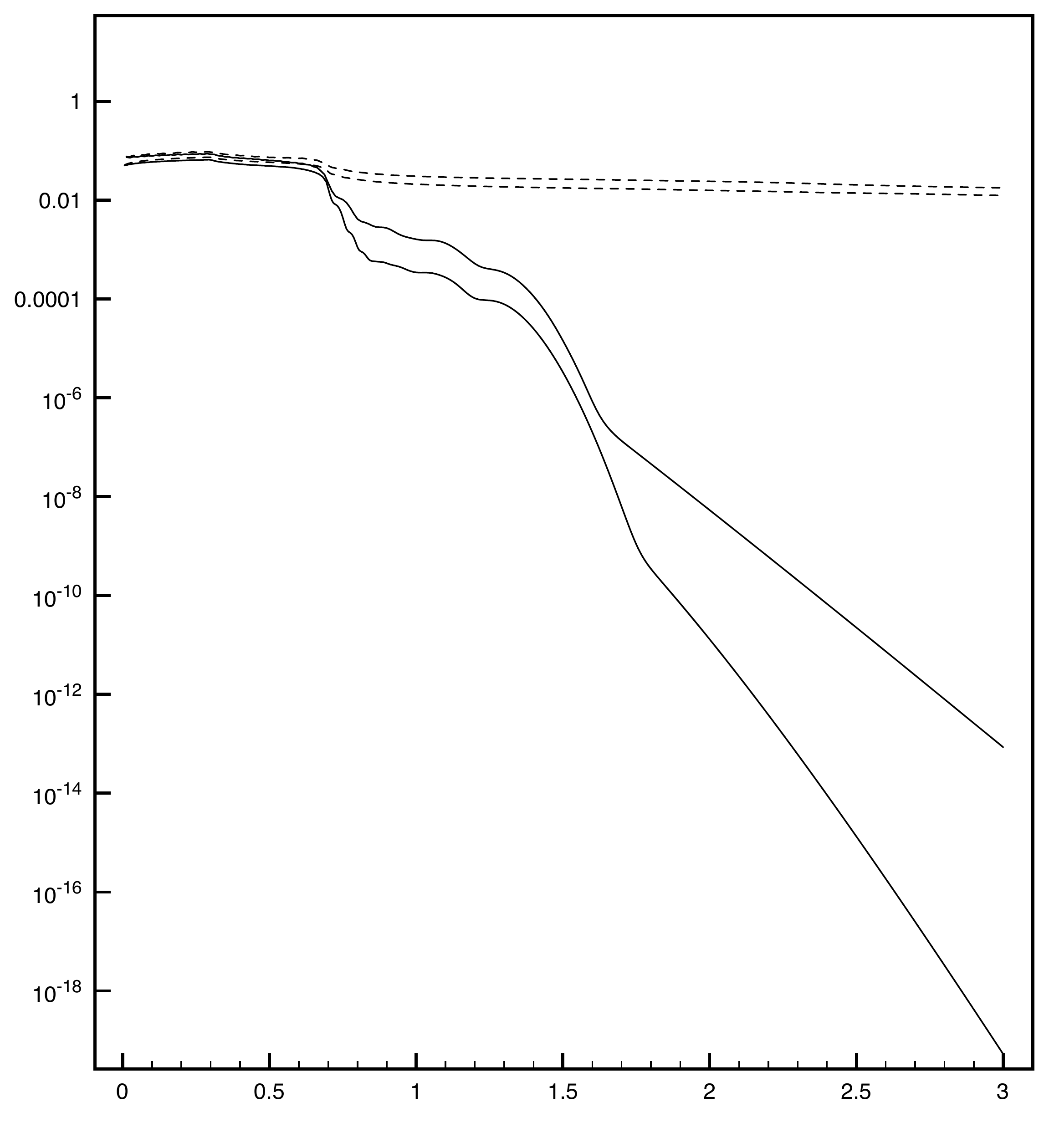}
\includegraphics[width=0.45\linewidth]{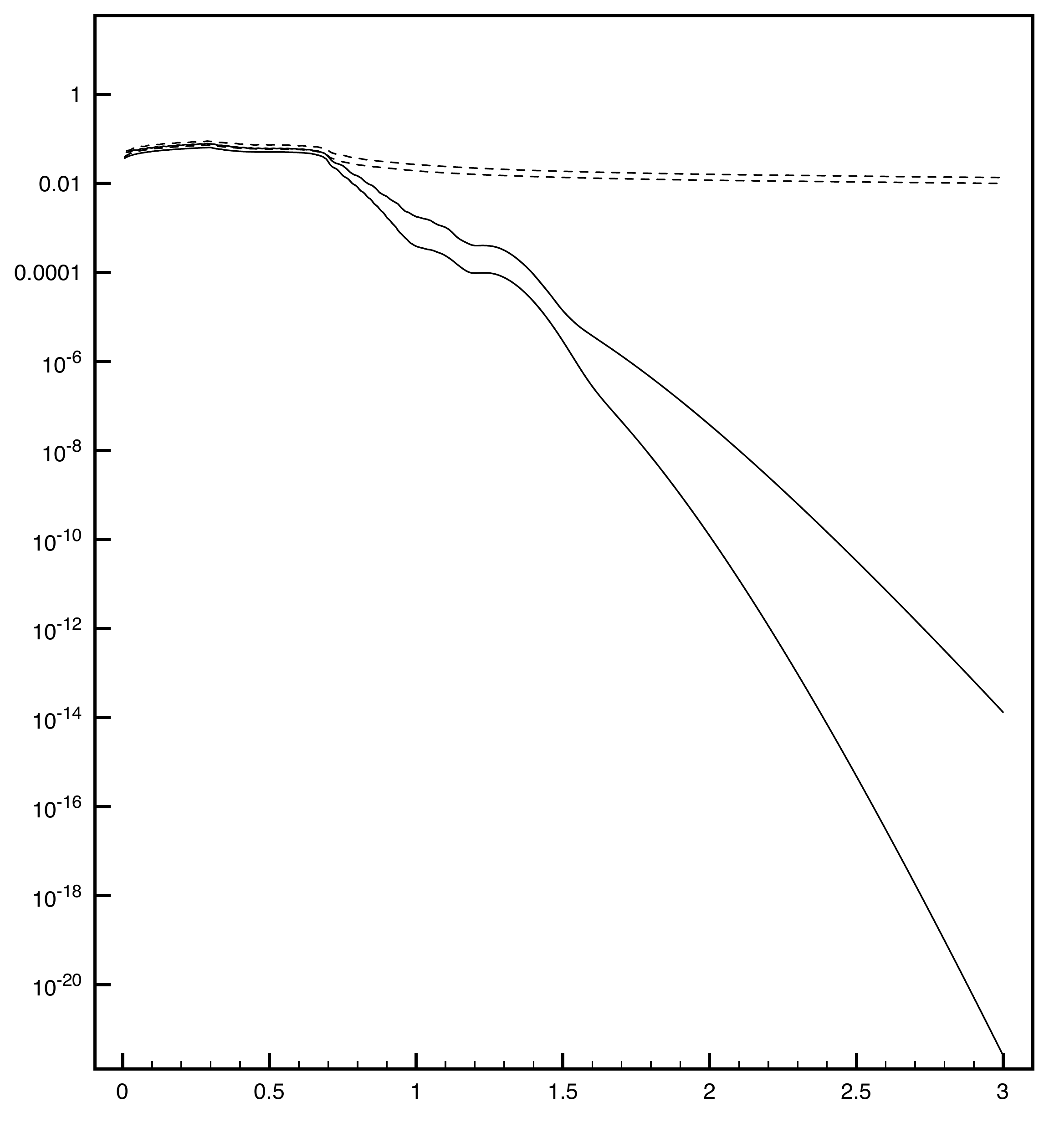}
\caption{Comparison of stabilized (full line) and unstabilized (dashed
  line) methods with $P_1$ (left) and $P_2$ (right)
  approximation. Evolution of the global $L^2$-error in time. Initial data from figure
  \ref{fig:data2}, left graphic. In each case the upper curve has
  $nele=40$ and the lower curve $nele=80$.}
\label{fig:longtime}
\end{figure}

\section*{Appendix}
Here we give the proofs of the approximation results for the
$L^2$-projection, Lemma \ref{lem:weight_approx} and Lemma
\ref{lem:super} and finally the weighted discrete interpolation result
\eqref{eq:stab_weight2}. 

First we give a simple super approximation result for
%$\pi_h$ using the Lemma \ref{lem:weight_approx} and 
the Lagrange interpolant $i_h$ that will be useful for the proofs of
inequalities \eqref{super1} and \eqref{super2}. For a general
discussion of discrete commutator properties we refer to \cite{Bert99}.
\begin{lemma}\label{lem:Lagrange_sup}
Let $\phi\in W^{k+1,\infty}(\Omega)$ satisfying \eqref{eq:varpi_bound} with $K>1$
and $h<1$. Then for $h^{\frac12}/K$
sufficiently small,
there holds for all $v_h \in V_h$, $\ele \in \mathcal{T}$,
\[
|\phi v_h - i_h (\phi v_h) |_{H^s(\ele)} \leq C h^{\frac12-s}/K \,\|
\phi v_h\|_{\ele},\quad 0 \leq s \leq 2.
\]
\end{lemma}
\begin{proof}
By the approximation properties of $i_h$
there holds
\begin{equation}\label{eq:interpol_Lag}
|\phi v_h - i_h (\phi v_h)|_{H^{s}(\ele)} \leq C h^{k+1-s} \|D^{k+1}
(\phi v_h)\|_{\ele}.
\end{equation}
Using the product rule and the fact that $D^{k+1} v_h = 0$ since $v_h
\vert_\ele \in \mathbb{P}_k(\ele)$, we see that
\[
\|D^{k+1} (\phi v_h)\|_{\ele} \leq C\sum_{l=1}^{k+1} |
\phi|_{W^{l,\infty}(\ele)} |v_h|_{H^{k+1-l}(\ele)}. 
\]
By applying the inverse inequality \eqref{eq:inverse_ineq} repeatedly 
the derivatives on $v_h$ can be eliminated at the price of factors of
the inverse of $h$,
\begin{equation}\label{eq:dphi_bound}
 h^{k+1-s} \|D^{k+1}
(\phi v_h)\|_{\ele} \leq C h^{1-s} \|v_h\|_{\ele} \sum_{l=1}^{k+1} h^{l-1} |
\phi|_{W^{l,\infty}(\ele)}.
\end{equation}
Using the bound \eqref{eq:varpi_bound} it then follows that
\begin{equation}\label{eq:phi_prop_bound}
 \sum_{l=1}^{k+1} h^{l-1} |
\phi|_{W^{l,\infty}(\ele)} \leq  C  \sum_{l=1}^{k+1} h^{l-1} (K h^{\frac12})^{-l}
\|\phi\|_{L^\infty(\ele)}  \leq C  (K h^{\frac12})^{-1}\|\phi\|_{L^\infty(\ele)}.
\end{equation}
Where we used the assumption that $h<1$ and $K>1$ in the last
inequality. 
Combining the bounds \eqref{eq:interpol_Lag}, \eqref{eq:dphi_bound} and
\eqref{eq:phi_prop_bound} it follows that
\[
\|\phi v_h - i_h (\phi v_h) \|_{H^{s}(\ele)} \leq C h^{1-s} (K h^{\frac12})^{-1} \|
\phi\|_{L^\infty(\ele)} \|v_h\|_{\ele}.
\]
The claim now follows by applying \eqref{eq:move_varpi}.
\end{proof}
\begin{proof}(Lemma \ref{lem:weight_approx}).
First note that by the construction of $\varpi$ there holds
\[
|\nabla \varpi| \leq C (\sqrt{h} K)^{-1} \varpi \leq (C \sqrt{h}/K)
h^{-1} \varpi
\]
and we see that we may apply \eqref{L2_approx1}-\eqref{L2_approx3} with
$\phi  = \varpi$ for $ (C \sqrt{h}/K)$ small enough.
\paragraph{\bf{Proof of \eqref{eq:move_varpi}}}
To prove \eqref{eq:move_varpi}, consider a triangle $\ele$, assume that
the max value in $\max_{(x,t) \in \ele\times  I_\delta} \varpi(x,t)$ is taken
at $(x^*,t^*)\in \ele\times I_\delta$. Then
\begin{multline*}
\max_{(x,t) \in \ele \times  I_\delta} \varpi(x,t) \|v\|_{\ele} = \|\varpi(x^*,t^*) v\|_{\ele} \leq
\|(\varpi(x^*,t^*) - \varpi(\cdot, \tilde t)) v\|_{\ele} + \| \varpi v\|_{\ele} \\
 \leq C h^{\frac12}
K^{-1} \varpi(x^*,t^*)  \|v\|_{\ele} + \| \varpi(\cdot, \tilde t) v\|_{\ele},
\end{multline*}
for any $\tilde t \in I_\delta$.
Assuming that $C h^{\frac12}
K^{-1} \leq \frac12$ we see that
\begin{equation}\label{eq:move_varpi_end}
\max_{(x,t) \in \ele\times  I_\delta} \varpi(x,t) \|v\|_{\ele} \leq 2 \|
\varpi(\cdot, \tilde t) v\|_{\ele}, \forall \tilde t \in I_\delta.
\end{equation}
\paragraph{\bf{Proof of \eqref{eq:error_est}}}
For the proof of \eqref{eq:error_est} first apply the stabilities
\eqref{L2_approx1}-\eqref{L2_approx2}. For the $L^2$-norm this yields
\[
\|\varpi (v - \pi_h v)\|_\Omega  \leq \|\varpi (v - i_h v)\|_\Omega +
\|\varpi \pi_h ( i_h v- v)\|_\Omega \leq  C \|\varpi (v - i_h v)\|_\Omega.
\]
Then apply interpolation locally
and \eqref{eq:move_varpi}.
\begin{equation*}
\|\varpi (v - i_h v)\|_{\ele} \leq  \max_{x \in \ele} \varpi(x) \|v - i_h
v\|_{\ele} 
\leq C \max_{x \in \ele} \varpi(x) h^{k+1} \|D^{k+1} v\|_{\ele} \leq 2 C h^{k+1} \|\varpi D^{k+1} v\|_{\ele}.
\end{equation*}
The claim follows by summing over $\ele \in \mathcal{T}$.
The bound on the $H^1$-norm is identical.
\paragraph{\bf{Proof of \eqref{eq:stab_est}}}
The stabilization operator is defined by the sum of the jumps over the
faces of the element of the gradient. The first step is to split that
jump using the triangle inequality over each face. Given a face $F = \partial S_1
\cap \partial S_2$ for elements $S_1$ and $S_2$ this takes the form.
\[
\|\jump{\nabla (v - \pi_h v)}\|_F^2 \leq  2 (\|\nabla (v - \pi_h
v)\|_{\partial S_1 \cap F}^2 + \|\nabla (v - \pi_h
v)\|_{\partial S_2 \cap F}^2).
\]
By breaking up the jumps on the
contributions from
respective element faces in this was we have
\[
s_{\varpi}(v - \pi_h v,v - \pi_h v) \leq C \sum_{\ele \in \mathcal{T}}
(\max_{x \in \ele} \varpi(x))^2 h^2
\beta_\infty \|\nabla (v- \pi_h v) \|_{\partial \ele}^2.
\]
Now apply the trace inequality \eqref{eq:trace} on each element to see that
\[
\|\nabla (v- \pi_h v) \|_{\partial \ele} \leq h^{\frac12} |\nabla (v-
\pi_h v) |_{H^1(\ele)} + h^{-\frac12} \|\nabla (v- \pi_h v) \|_{\ele}.
\]
For the first term in the right hand side add and subtract $i_h u$,
split it using a triangle inequality and
use an inverse inequality in one of the terms and interpolation in the other to
see that
\begin{multline*}
|\nabla (v-
\pi_h v) |_{H^1(\ele)} \leq |\nabla (v-
i_h v) |_{H^1(\ele)} + |\nabla (i_h v-
\pi_h v) |_{H^1(\ele)} \\
\leq C h^{k-1} \|D^{k+1} v\|_{\ele} + C h^{-1} \|\nabla (v-
\pi_h v) \|_{\ele}.
\end{multline*}
It follows using \eqref{eq:move_varpi} that
\[
\sum_{\ele \in \mathcal{T}} \varpi(x)^2 h^2
\beta_\infty\|\nabla (v- \pi_h v) \|^2_{\partial \ele} \leq
C  \beta_\infty h^{2k+1} \| D^{k+1} v\|^2_\varpi  + C \beta_\infty h \|\nabla (v-
\pi_h v) \|_{\varpi}^2.
\]
The claim now follows by applying \eqref{eq:error_est} to the second
term of the right hand side.
\end{proof}

\begin{proof}(Lemma \ref{lem:super})
\paragraph{\bf{Proof of \eqref{super1}}}
To prove \eqref{super1} recall that 
$$
|\nabla \varpi^{-1}| = |\varpi^{-2} \nabla \varpi| \leq C (\sqrt{h} K)^{-1}\varpi^{-1}
$$
and we may apply \eqref{L2_approx1} with
$\phi=\varpi^{-1}$ to get
\[
\|\varpi^{-1}(\varpi^2 v_h - \pi_h (\varpi^2 v_h))\|_\Omega \leq C
\|\varpi^{-1}(\varpi^2 v_h - i_h (\varpi^2 v_h))\|_\Omega.
\]
Consider one simplex $\ele$, take out the weight and then then apply Lemma
\ref{lem:Lagrange_sup} followed by \eqref{eq:move_varpi}
\[
\|\varpi^{-1}(\varpi^2 v_h - i_h (\varpi^2 v_h))\|_{\ele} \leq (\max_{x \in
  \ele} \varpi^{-1}) \|\varpi^2 v_h - i_h (\varpi^2 v_h)\|_{\ele} \leq C
h^{\frac12}/K \|\varpi v_h\|_{\ele}.
\]
Finally take the square of both sides and sum over the simplices. The
$H^1$-norm estimate follows using similar arguments.
\paragraph{\bf{Proof of \eqref{super2}}} For the
inequality \eqref{super2} we consider one element of the sum and apply
the trace inequality \eqref{eq:trace},
\begin{alignat}{1}\nonumber
\|\varpi^{-1}\nabla(\varpi^2 v_h - \pi_h (\varpi^2 v_h))\|_{\partial
  \ele} & \leq \max_{x \in \ele} \varpi^{-1} h^{\frac12}|\nabla(\varpi^2 v_h
- \pi_h (\varpi^2 v_h))|_{H^1(\ele)}  \\
& + \max_{x \in \ele} \varpi^{-1} h^{-\frac12}\|\nabla(\varpi^2 v_h
- \pi_h (\varpi^2 v_h))\|_{\ele}. \label{eq:2nd_trace_term}
\end{alignat}
In the first term, add and subtract $\nabla i_h (\varpi^2 v_h)$ and
use the triangle inequality followed by an inverse inequality to
obtain
\begin{multline*}
\max_{x \in \ele} \varpi^{-1} h^{\frac12}|\nabla(\varpi^2 v_h
- \pi_h (\varpi^2 v_h))|_{H^1(\ele)} \\
\leq C\max_{x \in \ele} \varpi^{-1}
h^{\frac12} (|\nabla(\varpi^2 v_h
- i_h (\varpi^2 v_h))|_{H^1(\ele)} + h^{-1} \|\nabla(i_h \varpi^2 v_h
- \pi_h (\varpi^2 v_h))\|_{\ele}).
\end{multline*}
For the first term in the right hand side we use Lemma
\ref{lem:Lagrange_sup}, with $s=2$,
\begin{equation}\label{eq:hterm}
h^{\frac12} \max_{x \in \ele} \varpi^{-1}|\nabla(\varpi^2 v_h
- i_h (\varpi^2 v_h))|_{H^1(\ele)} \leq C \max_{x \in \ele}
\varpi^{-1}K^{-1} h^{-1} \|\varpi^2 v_h\|_{\ele} \leq C K^{-1} h^{-1} \|\varpi v_h\|_{\ele}.
\end{equation}

To bound the second term we use \eqref{eq:move_varpi}, sum over $\ele
\in \mathcal{T}$
and use the stability of the $L^2$-projection \eqref{L2_approx2} to
get
\[
\sum_{\ele \in \mathcal{T}} (\max_{x \in \ele} \varpi^{-2})
h^{-1} \|\nabla(i_h \varpi^2 v_h
- \pi_h (\varpi^2 v_h))\|_{\ele}^2 \leq C h^{-1} \|\varpi^{-1} \nabla(i_h \varpi^2 v_h
- \varpi^2 v_h)\|_{\Omega}^2.
\]
We see that after summation over $\ele$ the second term in the right hand
side of
\eqref{eq:2nd_trace_term} also is on this form. 

On every $\ele$ take out the factor $\max_{x \in \ele} \varpi^{-1}$ and
apply Lemma \ref{lem:Lagrange_sup} followed by
\eqref{eq:move_varpi} to arrive at
\[
h^{-\frac12} \|\varpi^{-1} \nabla(i_h \varpi^2 v_h
- \varpi^2 v_h)\|_{\Omega}\leq C K^{-1} h^{-1} \|\varpi  v_h\|_{\Omega}
\]
which together with \eqref{eq:hterm}, summed over $\ele$, concludes the proof of \eqref{super2}.
\end{proof}
\begin{proof}(Inequality \eqref{eq:stab_weight2}).
For simplicity consider the form $\bfbeta \cdot \nabla u_h = \partial_x
u_h$. Using the product rule $\partial_x (\varpi^2 v_h) = (\partial_x
\varpi^2) v_h+ \varpi^2 \partial_x v_h $ and the triangle inequality it follows that
\begin{multline}\label{eq:first_step}
\|h^{\frac12} (\partial_x  (\varpi^2 v_h)
- \pi_h(\partial_x (\varpi^2 v_h)))\|_{\varpi^{-1}}^2 \leq
2h \|(\partial_x \varpi^2) v_h - \pi_h( \partial_x \varpi^2
v_h)\|_{\varpi^{-1}}^2 \\
+ 
2h \|(\varpi^2 \partial_x v_h)
- \pi_h(\varpi^2\partial_x v_h)\|_{\varpi^{-1}}^2.
\end{multline}
Noting that by the $L^2$-stability of $\pi_h$, the bound of $\varpi$,
Lemma \ref{lem:Lagrange_sup}, \eqref{eq:varpi_bound} and \eqref{eq:move_varpi}
\[
h \|(\partial_x \varpi^2) v_h - \pi_h( \partial_x \varpi^2
v_h)\|_{\varpi^{-1}}^2  \leq C h \|(\partial_x \varpi^2) v_h - i_h( \partial_x \varpi^2
v_h)\|_{\varpi^{-1}}^2  \leq C K^{-2} \|v_h\|^2_\varpi.
\]
It only remains to bound the second term of \eqref{eq:first_step}. We add and subtract $\pi_0 \varpi^2$ defined by 
\[
\pi_0 \varpi^2 \vert_{\ele} = |\ele|^{-1} \int_{\ele} \varpi^2
\]
and use the triangle inequality to obtain
\begin{alignat*}{1}
h \|(\varpi^2 \partial_x v_h)
- \pi_h(\varpi^2 \partial_x v_h)\|_{\varpi^{-1}}^2 & \leq
Ch \|( \varpi^2 \partial_x v_h - (\pi_0 \varpi^2) \partial_x v_h
)\|_{\varpi^{-1}}^2\\
& +C h \|(  (\pi_0 \varpi^2) \partial_x v_h
- \pi_h( (\pi_0 \varpi^2) \partial_x v_h )\|_{\varpi^{-1}}^2\\
&+ Ch \| ( \pi_h( (\pi_0 \varpi^2) \partial_x v_h )
- \pi_h(\varpi^2 \partial_x v_h))\|_{\varpi^{-1}}^2 = T_1+T_2+T_3.
\end{alignat*}
First, for $T_3$, observe that by the stability of the $L^2$-projection \eqref{L2_approx1}
we have
\begin{equation}\label{eq:T3}
h \| ( \pi_h( (\pi_0 \varpi^2) \partial_x v_h )
- \pi_h(\varpi^2 \partial_x v_h))\|_{\varpi^{-1}}^2 \leq C h \|( \varpi^2 \partial_x v_h - (\pi_0 \varpi^2) \partial_x v_h
)\|_{\varpi^{-1}}^2 \leq C T_1,
\end{equation}
so only $T_1$ and $T_2$ need to be bounded. For $T_1$, by the
approximation $\|\varpi^2 - \pi_0 \varpi^2\|_{L^\infty(\ele)} \leq C h^{\frac12}/K
\|\varpi\|_{L^\infty(\ele)}^2$ and \eqref{eq:move_varpi} we have for one
simplex $\ele$,
\[
\|\varpi^{-1} ( \varpi^2 \partial_x v_h - (\pi_0 \varpi^2) \partial_x v_h
)\|_{\ele} \leq  h^{\frac12}/K 
\max_{x \in \ele}\varpi^2 \max_{x \in \ele} \varpi^{-1} \|\partial_x v_h\|_{\ele}
\leq C h^{-\frac12} K^{-1} \|\varpi v_h\|_{\ele}.
\]
Taking the square of both sides and summing over all simplices yields the bound for $T_1$,
\begin{equation}\label{eq:T1}
h \|( \varpi^2 \partial_x v_h - (\pi_0 \varpi^2) \partial_x v_h
)\|_{\varpi^{-1}}^2 \leq C K^{-2} \|v_h\|_\varpi^2.
\end{equation}
Finally for the term $T_2$ we use \eqref{eq:stab_weight} with $\bfbeta_0
= (\pi_0 \varpi^2) e_x$. This leads to 
\begin{equation}\label{eq:T2b}
h \|(  (\pi_0 \varpi^2) \partial_x v_h
- \pi_h( (\pi_0 \varpi^2) \partial_x v_h )\|_{\varpi^{-1}}^2 \leq
C_{ws} s_{\varpi^{-1}}( (\pi_0 \varpi^2) v_h,(\pi_0 \varpi^2) v_h).
\end{equation}
Adding and subtracting $\varpi^2$ and using the triangle inequality
and the fact that $\varpi^2$ is smooth leads to 
\begin{equation}\label{eq:T2c}
 s_{\varpi^{-1}}( (\pi_0 \varpi^2) v_h,(\pi_0 \varpi^2) v_h) \leq
 2 s_{\varpi}(v_h,v_h) + 2 s_{\varpi^{-1}}( (\varpi^2 -\pi_0 \varpi^2)
 v_h,(\varpi^2 -\pi_0 \varpi^2) v_h).
\end{equation}
For the second term of the right hand side consider the boundary of
one triangle and
apply the trace inequality \eqref{eq:trace}, followed by the approximation of $\pi_0$ to get
\begin{equation*}
\|h (\varpi^2 -\pi_0 \varpi^2) \nabla v_h\|_{\partial \ele} \leq C \max_{x
  \in \ele} \varpi^2 K^{-1} h^{\frac32} (h^{\frac12} |\nabla v_h|_{H^1(\ele)}
+h^{-\frac12} \|\nabla v_h\|_{\ele}) \\
\leq C K^{-1} \|\varpi^2 v_h\|_{\ele}.
\end{equation*}
The last step followed using the inverse inequality \eqref{eq:inverse_ineq} and
\eqref{eq:move_varpi}. Proceeding by
applying the previous bound to all triangle faces, it follows that
\begin{multline}\label{eq:T2d}
 s_{\varpi^{-1}}( (\varpi^2 -\pi_0 \varpi^2)
 v_h,(\varpi^2 -\pi_0 \varpi^2) v_h)\leq
 C \sum_{\ele \in \mathcal{T}}  \max_{x
  \in \ele} \varpi^{-2} \|h (\varpi^2 -\pi_0 \varpi^2) \nabla
v_h\|_{\partial \ele}^2 \\
 \leq C\sum_{\ele \in \mathcal{T}}  \max_{x
  \in \ele} \varpi^{-2}  K^{-2} \|\varpi^2 v_h\|_{\ele}^2 \leq C K^{-2} \| v_h\|_\varpi^2
\end{multline}
where the last step follows using \eqref{eq:move_varpi}. The proof is
now finished by collecting the bounds \eqref{eq:T3} -- \eqref{eq:T2d}.
\end{proof}
\begin{proof}(Lemma \ref{lem:technical_tim}).
Using $\delta t \leq C h$ and \eqref{eq:varpi_bound} there holds
\begin{multline*}
 \|v_h\int_{t_{n-1}}^{t_n} \partial_t
  \varpi ~\mbox{d}t\|_\Omega \leq C\delta t/(K h^{\frac12}) \left( \sum_{\ele
    \in \mathcal{T}} \max_{(x,t) \in \ele \times [t_{n-1},t_n]}
  \varpi(x,t)^2 \|v_h \|_{\ele}^2\right)^{\frac12}\leq \delta t^{\frac12}   C/K \|v_h\|_{\varpi_{n}}.
\end{multline*}
For the second inequality we applied \eqref{eq:move_varpi} elementwise
and then upper bounded $\min_{t \in [t_{n-1},t_n]}
  \| v_h \varpi(\cdot,t)\|_S$ by $\|v_h\|_{\varpi_{n}}$.
For the bound of the second term 
observe that, estimating
\[
|\int_{t_{n-1}}^{t_n} \int_{t}^{t_n}\partial_{tt}
  \varpi^2 ~\mbox{d}s\,\mbox{d}t| \leq \delta t^2 \max_{t \in [t_{n-1},t_n] }|\partial_{tt}
  \varpi^2 |
\]
and then
applying \eqref{eq:varpi_bound} repeatedly with $l=1$ and $2$, to show
\[
\max_{t \in [t_{n-1},t_n] }|\partial_{tt}
  \varpi^2 | \leq C^2 h^{-1} K^{-2} \max_{t \in [t_{n-1},t_n] }\varpi^2.
\] 
It follows that for all $\ele \in \mathcal{T}$,
\[
\|v_h\left|\int_{t_{n-1}}^{t_n} \int_{t}^{t_n}\partial_{tt}
  \varpi^2 ~\mbox{d}s\,\mbox{d}t\right|^{\frac12}\|_{\ele} \leq \delta
t^{\frac12} C K^{-1}\max_{(x,t)
  \in \ele \times [t_{n-1},t_n] }\varpi \|v_h\|_{\ele}.
\]
Applying \eqref{eq:move_varpi} we conclude that
\[
\sum_{\ele
    \in \mathcal{T}} \|v_h\left|\int_{t_{n-1}}^{t_n} \int_{t}^{t_n}\partial_{tt}
  \varpi^2 ~\mbox{d}s\,\mbox{d}t\right|^{\frac12}\|_{\ele}^2 \leq \delta
t
 C^2 K^{-2} \sum_{\ele
    \in \mathcal{T}} \min_{t \in
  [t_{n-1},t_n]} \|v_h \varpi(\cdot,t)\|^2_{\ele} \leq \delta
t
 C^2 K^{-2} \|v_h\|^2_{\varpi_{n}}.
\]
\end{proof}
\section*{Declarations}
\begin{itemize}
\item[--]{\bf{Funding}}. The author acknowledges funding from EPSRC grants
  EP/P01576X/1 and EP/T033126/1.
\item[--] {\bf{Conflicts of interest/Competing interests.}} None.
\item[--] {\bf{Availability of data and material.}} The data used to
  produce figures can be made available upon
  reasonable request.
\item[--] {\bf{Code availability.}} Codes used to produce approximate
  solutions can be made available upon
  reasonable request.
\item[--] {\bf{Authors' contributions.}} N/A.
\end{itemize}
\bibliographystyle{plain}
\bibliography{references}
\end{document}